\documentclass[12pt]{article}
\usepackage[colorlinks]{hyperref}%The following 10 commandfs are for making a link to equation's number and citations
\usepackage{color}
\usepackage{graphicx}
\usepackage{graphics}
\usepackage{makeidx}
\usepackage{showidx}
\usepackage{latexsym}
\usepackage{amssymb}
\usepackage{verbatim}
\usepackage{amsmath}
\usepackage{amsthm}
\usepackage{amsfonts}
\usepackage{amssymb,amsmath}
\usepackage{latexsym,amsthm,amscd}
\usepackage[all]{xy}

\newtheorem{prop}{Proposition}[section]
\newtheorem{lemma}{Lemma}[section]
\newtheorem{defn}{Definition}[section]%If we dont put[thm] the numbering is done independently
\newcounter{alphthm}
\newtheorem{Lemma}[alphthm]{Lemma}

\newtheorem{cor}{Corollary}[section]

\newtheorem{rem}{Remark}[section]

\newtheorem{ex}{Example}[section]
\newtheorem{thm}{Theorem}
\newtheorem{lem}{Lemma}[section]
\newcommand{\be}{\begin{equation}}
\newcommand{\ee}{\end{equation}}
\newcommand{\ben}{\begin{enumerate}}
\newcommand{\een}{\end{enumerate}}
\newcommand{\beq}{\begin{eqnarray}}
\newcommand{\eeq}{\end{eqnarray}}
\newcommand{\beqn}{\begin{eqnarray*}}
\newcommand{\eeqn}{\end{eqnarray*}}

\newcommand{\bpf}{\begin{proof}}
\newcommand{\epf}{\end{proof}}
\newcommand{\bl}{\begin{lem}}
\newcommand{\el}{\end{lem}}
\newcommand{\bp}{\begin{prop}}
\newcommand{\ep}{\end{prop}}
\newcommand{\bd}{\begin{defn}}
\newcommand{\ed}{\end{defn}}
\newcommand{\bt}{\begin{thm}}
\newcommand{\et}{\end{thm}}
\newcommand{\R}{I\!\! R}
\def\nn{\nonumber}
\newcommand\bpr{\begin{prop}}
\newcommand\epr{\end{prop}}
\title{Ricci flow on Finsler surfaces}
\author{B. Bidabad\thanks{The corresponding author, bidabad@aut.ac.ir}\, and\, M.K. Sedaghat}
\date{}
\begin{document}
\maketitle
\noindent
\begin{abstract}
Here, we study the existence and uniqueness of solutions to the Ricci flow on Finsler surfaces and show short time existence of solutions for such flows. To this purpose, we first study the Finslerian Ricci-DeTurck flow on Finsler surfaces and find a unique short time solution to this flow. Then, we find a solution to the original Ricci flow by pulling back the solution of the Ricci-DeTurck flow using appropriate diffeomorphisms. At the end, we illustrate this argument by some examples.
 \end{abstract}
\vspace{.5cm}
%%%%%%%%%%%%%%%%%%%%%%%%%%%%%%%%%%%%%%%%%%%%%%%%%%%%%%%%%%%%%%%%%%%%%%%%%%%%%%%%%%%%%%%%%%%%%%%%%%%%%%%%%%%%
\section*{Introduction}
The \emph{Ricci flow} is a geometric evolution equation for the metric tensor on a general Riemannian manifold. The normalized Ricci flow has the property that all its fixed points are Einstein metrics. In his celebrated paper \cite{Ham1}, Hamilton showed that in 3-manifolds, the positive Ricci curvature condition on the initial metric implies that the Ricci flow exists for all time and converges to a Riemannian metric of constant curvature.
 This phenomenon has been shown later for other types of curvature conditions in other dimensions by several authors.

  The study of Ricci flow on surfaces is far simpler than its counterparts in higher dimensional cases.  Hence one can obtain much more detailed and comprehensive results.
   On surfaces, the Ricci flow solutions remain within a conformal class and clearly coincide with that of the \emph{Yamabe flow} on surfaces. In \cite{Ham2} Hamilton proved that for a compact oriented Riemannian surface $(M,g)$,  if $M$ is not diffeomorphic to the 2-sphere $\mathbb{S}^2$, then any metric $g$ converges to a constant curvature metric under the Ricci flow and if $M$ is diffeomorphic to $\mathbb{S}^2$, then any metric $g$ with positive Gaussian curvature on $\mathbb{S}^2$ converges to a metric of constant curvature under the flow.

Later, Chow in \cite{Chow2} removed the positive Gaussian curvature assumption in Hamilton's theorem and proved that for evolution of any metric on $\mathbb{S}^2$, under Hamilton's Ricci flow, the Gaussian curvature becomes positive in finite time and concluded that under the flow any metric $g$ on a Riemannian surface converges to a metric of constant curvature. Thus for compact surfaces, Ricci flow provides a new proof of the \emph{uniformization theorem}. Much is also known in the complete case. There are also many interesting subtleties in setting up this flow in the incomplete cases. Moreover, surface Ricci flow has started making impacts on practice fields and tackling fundamental engineering problems.

In Finsler geometry as a natural generalization of Riemannian geometry, the problem of constructing the Finslerian Ricci flow raises a number of new conceptual and fundamental issues in regards to the compatibility of geometrical and physical objects and their optimal configurations.

A fundamental step in the study of any system of evolutionary partial differential equations is to show the short time existence and uniqueness of solutions. Recently, an evolution of a family of Finsler metrics along Finsler Ricci flow has been studied by the first named author in several joint works and it has been shown that such flows exist in short time and converge to a limit metric; for instance, see \cite{YB2}.

In the present work, we study the Ricci flow on the closed Finsler surfaces and prove the short-time existence and uniqueness of solutions for the Ricci flow. Intuitively, since the Ricci flow system of equations is only weakly parabolic, its short-time existence and uniqueness do not follow from the standard theory of parabolic equations.  Following the procedure described by D. DeTurck in Riemannian space \cite{DeT}, we have introduced the Finslerian Ricci-DeTurck flow on Finsler surfaces by Eq. (\ref{22}) and prove existence and uniqueness of short-time solutions. More precisely, we prove:
\begin{thm}\label{main8}
Let $M$ be a compact Finsler surface. Given any initial Finsler structure $F_{0}$, there exists a real number $T>0$ and a smooth one-parameter family of Finsler structures $\tilde{F}(t)$, $t\in[0,T)$, such that $\tilde{F}(t)$ is a unique solution to the Finslerian Ricci-DeTurck flow with $\tilde{F}(0)=F_{0}$.
\end{thm}
 Next, a solution to the original Ricci flow is found by pulling back the solution to the Ricci-DeTurck flow via appropriate diffeomorphisms. This leads to
\begin{thm} \label{main14}
Let $M$ be a compact Finsler surface. Given any initial Finsler structure $F_{0}$, there exists a real number $T>0$ and a smooth one-parameter family of Finsler structures $F(t)$, $t\in[0,T)$, such that $F(t)$ is a unique solution to the Finslerian Ricci flow and $F(0)=F_{0}$.
\end{thm}

%As a first step to answer the Chern's question, we have considered evolution of a family of Finsler metrics, first under a general flow next under Finsler Ricci flow and prove that a family of Finsler metrics $g(t)$ which are solutions to the Finsler Ricci flow converge to a smooth limit Finsler metric as $t$ approaches the finite time $T$, see \cite{YB2}.
%Moreover, a Bonnet-Myers type theorem was studied and it is proved that on a Finsler space, a forward complete shrinking Ricci soliton is compact if and only if the corresponding vector field is bounded, see \cite{YB1}.
%, using which we have shown a compact shrinking Finsler Ricci soliton has a finite fundamental group and hence the first de Rham cohomology group vanishes, see \cite{YB1}.
%%%%%%%%%%%%%%%%%%%%%%%%%%%%
\section{Preliminaries and notations}
\subsection{Chern connection; A global approach}
Let $M$ be a real smooth surface and denote by $TM$ the tangent bundle of tangent vectors,  by  $\pi :TM_{0}\longrightarrow M$ the fiber bundle of non-zero tangent vectors and by $\pi^*TM\longrightarrow TM_0$ the pullback tangent bundle. Let $F$ be a Finsler structure on $TM_{0}$ and $g$ the related Finslerian metric. A \emph{Finsler manifold} is denoted here by the pair $(M,F)$. Any point of $TM_0$ is denoted by $z=(x,y)$, where $x=\pi z\in M$ and $y\in T_{x}M$.

Let us denote by $TTM_0$, the tangent bundle of $TM_0$ and by $\rho$, the canonical linear mapping $\rho:TTM_0\longrightarrow \pi^*TM,$ where, $ \rho=\pi_*$. For all $z\in TM_0$, $V_zTM$ is the set of all vertical vectors at $z$, that is, the set of vectors which are tangent to the fiber through $z$. Consider the decomposition $TTM_0=HTM\oplus VTM$, which permits us to uniquely represent a vector field $\hat{X}\in {\cal X}(TM_0)$ as the sum of the horizontal and vertical parts namely, $\hat{X}=H\hat{X}+V\hat{X}$. The corresponding basis is denoted here by $\{\frac{\delta}{\delta {x^i}},\frac{\partial}{\partial y^{i}}\}$, where, $\frac{\delta}{\delta {x^i}}:=\frac{\partial}{\partial x^{i}}-N_{i}^{j}\frac{\partial}{\partial y^{j}}$, $N^{j}_{i}=\frac{1}{2}\frac{\partial G^j}{\partial y^i}$ and $G^i$ are the spray coefficients defined by $G^{i}=\frac{1}{4}g^{ih}(\frac{\partial^{2}F^{2}}{\partial y^{h}\partial x^{j}}y^{j}-\frac{\partial F^{2}}{\partial x^{h}})$. We denote the \emph{formal Christoffel symbols} by $\gamma^{i}_{jk}=\frac{1}{2}g^{ih}(\partial_{j}g_{hk}+\partial_{k}g_{jh}-\partial_{h}g_{jk})$ where, $\partial_{k}=\frac{\partial}{\partial x^k}$. The dual bases are denoted by $\{dx^{i},\delta y^{i}\}$ where, $\delta y^{i}:=dy^{i}+N_{j}^{i}dx^{j}$. Let us denote a global representation of the Chern connection by $\nabla:{\cal X}(TM_0)\times\Gamma(\pi^{*}TM)\longrightarrow\Gamma(\pi^{*}TM)$. Consider the linear mapping
$\mu:TTM_0\longrightarrow \pi^*TM,$ defined by $\mu(\hat{X})=\nabla_{\hat{X}}{\bf y}$ where, $\hat{X}\in TTM_0$ and ${\bf y}=y^i\frac{\partial}{\partial x^i}$ is the canonical section of $\pi^*TM$.

The connection 1-forms of Chern connection in these bases are given by $\omega^{i}_{j}=\Gamma^{i}_{jk}dx^{k}$ where, $\Gamma^{i}_{jk}=\frac{1}{2}g^{ih}(\delta_{j}g_{hk}+\delta_{k}g_{jh}-\delta_{h}g_{jk})$ and $\delta_{k}=\frac{\delta}{\delta x^{k}}$. In the sequel, all the vector fields on $TM_0$ are decorated with a hat and denoted by $\hat{X}$, $\hat{Y}$, $\hat{Z}$ and the corresponding sections of $\pi^*TM$ by $X=\rho(\hat{X})$, $Y=\rho(\hat{Y})$ and $Z=\rho(\hat{Z})$, respectively unless otherwise specified.
%The structural equations of the Chern connection $\nabla$ are given by
%\begin{eqnarray*}
%&&\tau(\hat{X},\hat{Y})=\nabla_{\hat{X}}Y-\nabla_{\hat{Y}}X-\rho[\hat{X},\hat{Y}],\\
%&&\Omega(\hat{X},\hat{Y})Z=\nabla_{\hat{X}}\nabla_{\hat{Y}}Z-\nabla_{\hat{Y}}\nabla_{\hat{X}}Z
%-\nabla_{[\hat{X},\hat{Y}]}Z,
%\end{eqnarray*}
%where, $X=\rho(\hat{X})$, $Y=\rho(\hat{Y})$, $Z=\rho(\hat{Z})$ and $\hat{X}$, $\hat{Y}$ and $\hat{Y}$ are vector fields on $TM_0$.
The torsion freeness and almost metric compatibility of the Chern connection are given by
\begin{align}
&\nabla_{\hat{X}}Y-\nabla_{\hat{Y}}X=\rho[\hat{X},\hat{Y}],\label{tori}\\
&(\nabla_{\hat{Z}}g)(X,Y)=2C(\mu(\hat{Z}),X,Y),\label{gcomp}
\end{align}
respectively, where $C$ is the Cartan tensor with the components $C_{ijk}=\frac{\partial g_{ij}}{\partial y^{k}}.$ In a local coordinates on $TM$ the \emph{Chern horizontal} and \emph{vertical covariant derivatives} of an arbitrary $(1,2)$ tensor field $S$ on $\pi^{*}TM$ with the components $(S^{i}_{jk}(x,y))$ on $TM$ are denoted by
\begin{align*}
&\nabla_{l}S^{i}_{jk}:= \delta_{l}S^{i}_{jk}-S^{i}_{s k}\Gamma^{s}_{jl}-S^{i}_{js}\Gamma^{s}_{kl}+S^{s}_{jk}\Gamma^{i}_{s l}, %\label{1}
 \\
&\dot{\nabla}_{l}S^{i}_{jk}:=\dot{\partial}_{l}S^{i}_{jk},%\label{2}
\end{align*}
where, $\nabla_{l}:=\nabla_{\frac{\delta}{\delta x^l}}$ and $\dot{\nabla}_{l}:=\nabla_{\frac{\partial}{\partial y^l}}$. Horizontal metric compatibility of the Chern connection is given in local coordinates  by $\nabla_{l}g_{jk}=0$, see \cite[p.\ 45]{BCS}. The local \emph{Chern $hh$-curvature} tensor is given by
\begin{equation} \label{77}
R^{\,\,i}_{j\,\,kl}=\delta_{k}\Gamma^{i}_{\,jl}-\delta_{l}\Gamma^{i}_{\,jk}+
\Gamma^{i}_{\,hk}\Gamma^{h}_{\,jl}-\Gamma^{i}
_{\,hl}\Gamma^{h}_{\,jk},
\end{equation}
see \cite[p.\ 52]{BCS}. The \emph{reduced $hh$-curvature} tensor is a connection free tensor field which is also referred to as the \emph{Riemann curvature} by certain authors. In a local coordinates on $TM$, the components of the reduced $hh$-curvature tensor are given by $R^{i}_{\,\,k}:=\frac{1}{F^2}y^{j}R^{\,\,i}_{j\,\,km}y^{m}$, which are entirely expressed in terms of $x$ and $y$ derivatives of spray coefficients $G^{i}$ as follows
\begin{equation} \label{18}
R^{i}_{\,\,k}:=\frac{1}{F^2}(2\frac{\partial G^{i}}{\partial x^{k}}-\frac{\partial^{2}G^{i}}{\partial x^{j}\partial y^{k}}y^{j}+2G^{j}\frac{\partial^{2}G^{i}}{\partial y^{j}\partial y^{k}}-\frac{\partial G^{i}}{\partial y^{j}}\frac{\partial G^{j}}{\partial y^{k}}),
\end{equation}
see \cite[p.\ 66]{BCS}.
%Let  $c:I\longrightarrow M$ be an oriented $C^{\infty}$ parametric curve on $(M,F)$ with the parametric equation $x^{i}(t)$. We denote the  \emph{Chern covariant derivative along $c$} of a section $X$ of $\pi^{*}TM$ by $\nabla_{\dot{c}}X:=\frac{\delta X^{i}}{dt}\frac{\partial}{\partial x^{i}}$, where
%\begin{equation} \label{8}
%\frac{\delta X^{i}}{dt}=\frac{dX^{i}}{dt}+\Gamma^{i}_{kh}X^{k}\frac{dx^{h}}{dt},
%\end{equation}
%see \cite{BS} or \cite{BCS}, page 260.
\subsection{Lie derivatives of Finsler metrics}
The Lie derivative of an arbitrary Finslerian $(0,2)$ tensor field ${\cal T}={\cal T}_{jk}(x,y)dx^{j}\otimes dx^{k}$ on $\otimes^{2}\pi^{*}TM$ with respect to an arbitrary vector field $\hat V$ on $TM_0$ is given by
\beq
(\mathcal{L}_{\hat{V}}{\cal T})(X,Y)=\hat{V}({\cal T}(X,Y))-{\cal T}(\rho[\hat V,\hat X],Y)-{\cal T}(X,\rho[\hat V,\hat Y]),\nonumber
\eeq
where, $\rho(\hat{X})=X$, $\rho(\hat{Y})=Y$ and $\hat{X},\hat{Y}\in T_{z}TM_{0}$, see \cite{JB}. The Lie derivative of Finsler metric $g$ with respect to the arbitrary vector field $\hat V$ on $TM_0$ is given by
\beq
(\mathcal{L}_{\hat{V}}{g})(X,Y)=\hat{V}({g}(X,Y))-{g}(\rho[\hat V,\hat X],Y)-{g}(X,\rho[\hat V,\hat Y]).\nonumber
\eeq
By means of the torsion freeness of Chern connection defined by (\ref{tori}), Lie derivative of the Finsler metric $g$ can be rewritten as
\begin{align}
(\mathcal{L}_{\hat{V}}{g})(X,Y)&=\hat{V}(g(X,Y))-g(\nabla_{\hat{V}}X-\nabla_{\hat{X}}V,Y)-g(X,\nabla_{\hat{V}}Y-\nabla_{\hat{Y}}V)\nonumber\\
&=\hat{V}(g(X,Y))-g(\nabla_{\hat{V}}X,Y)+g(\nabla_{\hat{X}}V,Y)\nonumber\\
&\quad-g(X,\nabla_{\hat{V}}Y)+g(X,\nabla_{\hat{Y}}V).\label{Liederiv}
\end{align}
By  the almost $g$-compatibility of Chern connection defined by (\ref{gcomp}), we have
\begin{equation*}
2C(\mu(\hat{V}),X,Y)=(\nabla_{\hat{V}}g)(X,Y)=\hat{V}(g(X,Y))-g(\nabla_{\hat{V}}X,Y)-g(X,\nabla_{\hat{V}}Y),
\end{equation*}
Therefore,
\begin{equation}\label{campatibility}
\hat{V}(g(X,Y))=2C(\mu(\hat{V}),X,Y)+g(\nabla_{\hat{V}}X,Y)+g(X,\nabla_{\hat{V}}Y).
\end{equation}
Plugging the equation (\ref{campatibility}) in (\ref{Liederiv}) we obtain
\begin{equation}\label{finalliederiv}
(\mathcal{L}_{\hat{V}}{g})(X,Y)=2C(\mu(\hat{V}),X,Y)+g(\nabla_{\hat{X}}V,Y)+g(X,\nabla_{\hat{Y}}V).
\end{equation}
Replacing $X$ and $Y$ by the canonical section ${\bf y}=y^{i}\frac{\partial}{\partial x^{i}}$ in (\ref{finalliederiv}) we obtain
\begin{equation*}
(\mathcal{L}_{\hat{V}}{g})({\bf y},{\bf y})=2C(\mu(\hat{V}),{\bf y},{\bf y})+g(\nabla_{\hat{{\bf y}}}V,{\bf y})+g({\bf y},\nabla_{\hat{{\bf y}}}V),
\end{equation*}
where, $\hat{{\bf y}}=y^{i}\frac{\delta}{\delta x^{i}}$. Using $C(\mu(\hat{V}),{\bf y},{\bf y})=0$, see \cite[p.\ 23]{BCS}, and the symmetric property of $g(\nabla_{\hat{{\bf y}}}V,{\bf y})$ one arrives at
\begin{equation}\label{global}
(\mathcal{L}_{\hat{V}}{g})({\bf y},{\bf y})=2g({\bf y},\nabla_{\hat{{\bf y}}}V),
\end{equation}
where, $V=v^i\frac{\partial}{\partial x^i}$ is a section of $\pi^{*}TM$. In the local coordinates, (\ref{global}) can be written as
\begin{equation*}
y^iy^j\mathcal{L}_{\hat{V}}g_{ij}=2y^iy^jg_{ik}\nabla_{j}v^{k}.
\end{equation*}
Using $\nabla_{j}g_{ik}=0$, we obtain
\begin{equation}\label{FINAL}
y^iy^j\mathcal{L}_{\hat{V}}g_{ij}=2y^iy^j\nabla_{j}v_{i},
\end{equation}
where, $v_{i}=g_{ik}v^{k}$.
%%%%%%%%%%%%%
\subsection{The Berwald frame and a geometrical setup on $SM$}
Let $(M,F)$ be a Finsler surface and $SM$ the quotient of $TM_0$ under the following equivalence relation: $(x,y)\sim(x,\tilde{y})$ if and only if $y, \tilde{y}$ are positive multiples of each other. In other words, $SM$ is the bundle of all directions or rays, and is called the (projective) \emph{sphere bundle}. The local coordinates $(x^1,x^2)$ on $M$ induce the global coordinates $(y^1,y^2)$ on each fiber $T_{x}M$, through the expansion $y=y^{i}\frac{\partial}{\partial x^i}$. Therefore $(x^i;y^i)$ is a coordinate system on $SM$, where the coordinates $y^i$ are regarded as homogeneous coordinates in the projective space. Using the canonical projection $p:SM\longrightarrow M$, one can pull the tangent bundle $TM$ back to $p^{*}TM$ which is a vector bundle with the fiber dimension 2 over the 3-manifold $SM$. The vector bundle $p^{*}TM$ has a global section $l:=\frac{y^i}{F(y)}\frac{\partial}{\partial x^i}$ and a natural Riemannian metric which we here denote by $g:=g_{ij}(x,y)dx^{i}\otimes dx^{j}$. One can complete $l$ into a positively oriented $g$-orthonormal frame $\{e_1,e_2\}$ for $p^{*}TM$, with $e_{2}:=l$, by setting
\begin{align*}
&e_{1}=\frac{F_{y^2}}{\sqrt{g}}\frac{\partial}{\partial x^1}-\frac{F_{y^1}}{\sqrt{g}}\frac{\partial}{\partial x^2},\\
&e_{2}=\frac{y^1}{F}\frac{\partial}{\partial x^1}+\frac{y^2}{F}\frac{\partial}{\partial x^2},
\end{align*}
where, $\sqrt{g}:=\sqrt{\det(g_{ij})}$ and $F_{y^i}$ abbreviates the partial derivative $\frac{\partial F}{\partial y^i}$. In 2-dimensional
case, $\{e_1,e_2\}$ is a globally defined $g$-orthonormal
frame field for $p^{*}TM$ called a \emph{Berwald frame}. The natural dual of $l$ is the Hilbert form defined by $\omega:=F_{y^i}dx^i$, which is a global section of $p^{*}T^{*}M$. The coframe corresponding to $\{e_1,e_2\}$ is defined here by $\{\omega^1,\omega^2\}$, where
\begin{align}\label{dualbase}
&\omega^1=\frac{\sqrt{g}}{F}(y^2dx^1-y^1dx^2)=v^{1}_{1}dx^1+v^{1}_{2}dx^2,\nonumber\\
&\omega^2=F_{y^1}dx^1+F_{y^2}dx^2=v^{2}_{1}dx^1+v^{2}_{2}dx^2.
\end{align}
The sphere bundle $SM\subset TM$ is a 3-dimensional Riemannian manifold equipped with the induced \emph{Sasaki metric}
\begin{equation*}
\omega^1\otimes\omega^1+\omega^2\otimes\omega^2+
\omega^3\otimes\omega^3,
\end{equation*}
where,
\begin{equation}\label{dualbase1}
\omega^{3}:=\frac{\sqrt{g}}{F}(y^2\frac{\delta y^1}{F}-y^1\frac{\delta y^2}{F})=v^{1}_{1}\frac{\delta y^1}{F}+v^{1}_{2}\frac{\delta y^2}{F}.
\end{equation}
The collection $\{\omega^1,\omega^2,\omega^3\}$ is a globally defined orthonormal frame for $T^{*}(SM)$. Its natural dual frame is given by $\{\hat{e}_1,\hat{e}_2,\hat{e}_3\}$, where
\begin{align}
&\hat{e}_{1}=\frac{F_{y^2}}{\sqrt{g}}\frac{\delta}{\delta x^1}-\frac{F_{y^1}}{\sqrt{g}}\frac{\delta}{\delta x^2}=u^{1}_{1}\frac{\delta}{\delta x^1}+u^{2}_{1}\frac{\delta}{\delta x^2},\label{base1}\\
&\hat{e}_{2}=\frac{y^1}{F}\frac{\delta}{\delta x^1}+\frac{y^2}{F}\frac{\delta}{\delta x^2}=u^{1}_{2}\frac{\delta}{\delta x^1}+u^{2}_{2}\frac{\delta}{\delta x^2},\label{base2}\\
&\hat{e}_{3}=\frac{F_{y^2}}{\sqrt{g}}F\frac{\partial}{\partial y^1}-\frac{F_{y^1}}{\sqrt{g}}F\frac{\partial}{\partial y^2}=Fu^{1}_{1}\frac{\partial}{\partial y^1}+Fu^{2}_{1}\frac{\partial}{\partial y^2}.\label{base3}
\end{align}
These three vector fields on $SM$ form a global orthonormal frame for $T(SM)$. The first two are horizontal while the third one is vertical. The objects $\omega^1,\omega^2,\omega^3$ and $\hat{e}_1,\hat{e}_2,\hat{e}_3$ are defined in terms of objects that live on the slit tangent bundle $TM_{0}$. But they are invariant under positive rescaling in $y$. Therefore they give bonafide objects on the sphere bundle $SM$, see \cite[p.\ 92-94]{BCS}.
%%%%%%%%%%%%%%%%%%%%%%%%%%%%
\subsection{The  integrability condition for Finsler metrics}
% \setcounter{footnote}{0}
%*******************************************************************************************%%*******************************************************************************************************
Here, we first recall that all the Riemannian metrics on the fibers of the pulled-back
bundle do not come from a Finsler structure.
Hence, not every arbitrary symmetric $(0,2)$-tensor $g_{ij}(x,y)$ arises from a Finsler structure $F(x,y)$.
 %Therefore we have to incorporate a criterion insuring this fact
%into every step of our process.
 Intuitively, in order to make sure $g_{ij}(x,y)$ are components of a Finsler structure,
%as explained in [Bao, Robles 2004],
the essential integrability criterion is the total symmetry of $(g_{ij})_{y^k}$
on all three indices $i, j, k$. In fact, $g_{ij}(x,y)$ arises from a Finsler structure $F(x,y)$ if and only if $(g_{ij})_{y^k}$  is totally symmetric in its three indices, see \cite[p. 56]{Bao}. Symmetry of ${({g_{ij}})_{{y^k}}}$ on all three indices $i, j, k$ is known in the literature as \emph{integrability condition}. Moreover, we have to make sure the integrability criterion is satisfied in every step along the Ricci flow.
To this end we consider a general evolution equation given by
\begin{equation}\label{AR}
\frac{\partial}{\partial t}g(t)=\omega(t),\quad g(0):=g_{0},
\end{equation}
where, $\omega(t):=\omega(t, x, y)$ is a family of symmetric $(0,2)$-tensors on $\pi^{*}TM$, zero-homogenous with respect to $y$. The following Lemma establishes the integrability condition, see also \cite[p. 749]{YB2}.
\begin{lemma}\label{RE2}
Let $g(t)$ be a solution  to the evolution equation (\ref{AR}).
 There is a family of Finsler structures $F(t)$ on $TM$ such that,
\begin{equation}\label{Eq;IntCond}
g_{ij}(t)=\frac{1}{2}\frac{\partial^2 F(t)}{\partial y^i\partial y^j}.
\end{equation}
\end{lemma}
 \bpf
 Let $M$ be a compact differential manifold, $F(t)$ a family of smooth 1-parameter Finsler structures on $TM_0$ and $g(t)$ the Hessian matrix of $F(t)$ which defines a scalar product on $\pi^{*}TM$ for every $t$.
 Let $g(t)$ be a solution to the evolution equation (\ref{AR}). We have
\begin{equation}\label{ARR}
g(t)=g(0)+\int_{0}^{t}\omega(\tau)d\tau, \qquad \forall\tau\in[0,t).
\end{equation}
We show that the metric $g(t)$ satisfies the integrability condition, or equivalently there is a Finsler structure $F(t)$ on $TM_0$ satisfying \eqref{Eq;IntCond}.
%\begin{equation*}
%g_{ij}=\frac{1}{2}\frac{\partial^2 F}{\partial y^i\partial y^j}.
%\end{equation*}
For this purpose, we multiply  $g_{ij}$ by $y^i$ and $y^j$ in (\ref{ARR}),
\begin{equation*}
y^iy^jg_{ij}(t)=y^iy^jg_{ij}(0)+\int_{0}^{t}y^iy^j\omega_{ij}(\tau)d\tau.
\end{equation*}
By means of the initial condition $y^iy^jg_{ij}(0)= F^2(0)$, we get
\begin{equation}\label{EQ1}
y^iy^jg_{ij}(t)=F^2(0)+\int_{0}^{t}y^iy^j\omega_{ij}(\tau)d\tau.
\end{equation}
By positive definiteness assumption of $g_{ij}$, we put $F = (y^iy^jg_{ij})^{\frac{1}{2}}$. Twice vertical derivatives of (\ref{EQ1}) yields
\begin{equation}\label{EQ2}
\frac{1}{2}\frac{\partial^2 F^2}{\partial y^k\partial y^l}=g_{kl}(0)+\frac{1}{2}\int_{0}^{t}\frac{\partial^2}{\partial y^k\partial y^l}(y^iy^j\omega_{ij}(\tau))d\tau.
\end{equation}
On the other hand, by straightforward calculation we have
\begin{equation}\label{EQ3}
\frac{1}{2}\frac{\partial^2}{\partial y^k\partial y^l}(y^iy^j\omega_{ij}(\tau))=\frac{1}{2}\frac{\partial^2\omega_{ij}(\tau)}{\partial y^k\partial y^l}y^iy^j=\Big(\frac{\partial\omega_{ik}(\tau)}{\partial y^l}-\frac{\partial\omega_{il}(\tau)}{\partial y^k}\Big)y^i+\omega_{kl}(\tau),
\end{equation}
for all $\tau\in[0,t)$. Using (\ref{AR}) we obtain
\begin{equation*}
\frac{1}{2}\frac{\partial^2\omega_{ij}(\tau)}{\partial y^k\partial y^l}y^iy^j=0,\quad\frac{\partial\omega_{ik}(\tau)}{\partial y^l}=0,\quad\frac{\partial\omega_{il}(\tau)}{\partial y^k}=0.
\end{equation*}
Therefore, (\ref{EQ3}) is reduced to
\begin{equation}\label{EQ4}
\frac{1}{2}\frac{\partial^2}{\partial y^k\partial y^l}(y^iy^j\omega_{ij}(\tau))=\omega_{kl}(\tau),
\end{equation}
for all $\tau\in[0,t)$. Finally, replacing (\ref{EQ4}) in (\ref{EQ2}) we get
\begin{equation*}
\frac{1}{2}\frac{\partial^2 F^2}{\partial y^k\partial y^l}=g_{kl}(0)+\int_{o}^{i}\omega_{kl}(\tau)d\tau=g_{kl}.
\end{equation*}
Therefore, every $g_{ij}(t)$ on the fibers of pulled-back bundle, arises from a Finsler structure. This completes the proof.
%\hspace{\stretch{1}}$\Box$
\epf
%%%%%%%%%%%%%%%%%%%%%%%%%%%
%%%%%%%%%%%%%%%%%%%%%%%%%%%%
\section{Semi-linear strictly parabolic equations on $SM$}
Recall that a \emph{quasi-linear} system is a system of partial differential equations where, the derivatives of principal order terms occur only linearly and  coefficients may depend on derivatives of the lower order terms. It is called \emph{semi-linear} if it is quasi-linear and coefficients of the principal order terms depend only on the independent variables, but not on the solution, see \cite[p.\ 45]{Rog}. Let $M$ be a 2-dimensional manifold and $u:M\longrightarrow \mathbb{R}$ a smooth function on $M$. A \emph{semi-linear strictly parabolic} equation is a PDE of the form
\begin{eqnarray*}
\frac{\partial u}{\partial t}=a^{ij}(x,t)\frac{\partial^2 u}{\partial x^i\partial x^j}+h(x,t,u,\frac{\partial u}{\partial x^i}),\qquad i,j=1,2,
\end{eqnarray*}
where, $a^{ij}$ and $h$ are smooth functions on $M$ and for some constant $\lambda>0$ we have the parabolic assumption
\begin{eqnarray*}
a^{ij}\xi_{i}\xi_{j}\geq \lambda\parallel \xi\parallel^{2},\quad 0\neq\xi\in\chi(M),
\end{eqnarray*}
that is, all eigenvalues of $A=(a^{ij})_{2\times2}$ have positive signs or equivalently $A$ is positive definite.
\begin{defn}\label{semipar}
Let $M$ be a surface and $\phi:SM\longrightarrow \mathbb{R}$ a smooth function on the sphere bundle $SM$. Consider the following semi-linear strictly parabolic equation on $SM$;
\begin{equation*}%\label{parabolic}
\frac{\partial \phi}{\partial t}=G^{_{AB}}(x,y,t)\hat{e}_{_A}\hat{e}_{_B}\phi+h(x,y,t,\phi,\hat{e}_{_A}\phi),\qquad A,B=1,2,3 ,
\end{equation*}
where, $\hat{e}_{_A}$ is a local frame for the tangent bundle $TSM$ and stand for partial derivatives on $SM$. Here, $G^{_{AB}}$ and $h$ are smooth functions on $SM$ and $G=(G^{_{AB}})$ is positive definite.
\end{defn}
More precisely, a semi-linear strictly parabolic equation on $SM$ can be written in the form
\begin{equation}\label{po}
\frac{\partial \phi}{\partial t}=p^{ab}(x,y,t)\hat{e}_{a}\hat{e}_{b}\phi+q(x,y,t)\hat{e}_{3}
\hat{e}_{3}\phi+m^{a}(x,y,t)\hat{e}_{a}\hat{e}_{3}\phi+\textrm{lower order terms},
\end{equation}
where $a,b=1,2$, and the matrix
\begin{displaymath}
G=\left(\begin{array}{c|c}
P\,\,\, & \frac{1}{2}M \\
\hline\,\,
\frac{1}{2}M^{t}\,\,\ & Q
\end{array}\right)_{3\times3},
\end{displaymath}
is positive definite where, $P=(p^{ab})_{2\times2},Q=(q)_{1\times1}, M=(m^{a})_{2\times1}$.
\begin{lem}\label{mm}
Let $(M,F)$ be a Finsler surface and $\phi:TM\longrightarrow \mathbb{R}$ a zero-homogeneous smooth function on the tangent bundle $TM$. The semi-linear differential equation
\begin{equation}\label{lili}
\frac{\partial \phi}{\partial t}=g^{ij}\frac{\delta^2\phi}{\delta x^i\delta x^j}+F^2g^{ij}\frac{\partial^2 \phi}{\partial y^i\partial y^j}+\textrm{lower order terms},\qquad i,j=1,2,
\end{equation}
is a strictly parabolic equation on $SM$.
\end{lem}
\begin{proof}
Let us denote again by $\phi$ its restriction on $SM$. According to (\ref{base1}) and (\ref{base2}), replacing $\hat{e}_{a}=u^i_a\frac{\delta}{\delta x^i}$, we obtain
\begin{align*}
&\hat{e}_{a}\phi=u^i_a\frac{\delta \phi}{\delta x^i},\\
&\hat{e}_{b}\hat{e}_{a}\phi=u^{i}_{a}u^{j}_{b}\frac{\delta^2 \phi}{\delta x^i\delta x^j}+\hat{e}_b(u^{i}_{a})(\frac{\delta \phi}{\delta x^i}),\qquad a,b=1,2.
\end{align*}
Multiplying the both sides by $g^{ab}$ leads to
\begin{align*}
g^{ab}\hat{e}_{b}\hat{e}_{a}\phi&=g^{ab}u^{i}_{a}u^{j}_{b}\frac
{\delta^2 \phi}{\delta x^i\delta x^j}+g^{ab}\hat{e}_b(u^{i}_{a})
(\frac{\delta \phi}{\delta x^i})\\
&=g^{ij}\frac{\delta^2 \phi}{\delta x^i\delta x^j}+g^{ab}\hat{e}_b(u^{i}_{a})(\frac{\delta \phi}{\delta x^i}),
\end{align*}
where, $g_{ab}=g_{ij}u^i_au^j_b$. According to (\ref{dualbase}), by using the notations $\omega^{c}:=v^{c}_{i}dx^{i}$ and $B^c:=v^c_ig^{ab}\hat{e}_{b}(u^i_a)$ one can rewrite the expression
 $g^{ij}\frac{\delta^2 \phi}{\delta x^i\delta x^j}$ on $SM$ with respect to $\hat{e}_{a}$ as follows
\begin{equation*}%\label{First}
g^{ab}\hat{e}_{b}\hat{e}_{a}\phi-B^c\hat{e}_c\phi=g^{ij}\frac{\delta^2 \phi}{\delta x^i\delta x^j}, \qquad c=1,2.
\end{equation*}
Hence, (\ref{base3}) yields
\begin{align*}
\hat{e}_{3}\phi&=Fu^{i}_{1}\frac{\partial \phi}{\partial y^i},\\
\hat{e}_{3}\hat{e}_{3}\phi&=\hat{e}_{3}(Fu^{i}_{1}\frac{\partial \phi}{\partial y^i})
\nn\\&=F^2u^j_1 u^i_1\frac{\partial^2 \phi}{\partial y^j\partial y^i}+F(\hat{e}_{3} u^i_{1})(\frac{\partial \phi}{\partial y^i})+Fu^{j}_{1}(\frac{\partial F}{\partial y^j})u^i_1\frac{\partial \phi}{\partial y^i}.
\end{align*}
Using the fact $u^{j}_{1}\frac{\partial F}{\partial y^j}=0$, see \cite[p.\ 161]{HAZ}, we have
\begin{eqnarray*}
\hat{e}_{3}\hat{e}_{3}\phi=F^2u^j_1 u^i_1\frac{\partial^2
\phi}{\partial y^j\partial y^i}+F(\hat{e}_{3} u^i_{1})(\frac{\partial \phi}{\partial y^i}).
\end{eqnarray*}
Multiplying the both sides by $g^{11}$ and taking into account that $g^{11}u^{i}_{1}u^{j}_{1}=g^{ij}-y^iy^j$ we get
\begin{eqnarray*}
g^{11}\hat{e}_{3}\hat{e}_{3}\phi=F^2g^{ij}\frac{\partial^2 \phi}{\partial y^j\partial y^i}+Fg^{11}(\hat{e}_{3}u^i_{1})\frac{\partial\phi}{\partial y^i},
\end{eqnarray*}
where, $g_{11}=g_{ij}u^i_1 u^j_1$. According to (\ref{dualbase1}), we have $\omega^{3}=v^{1}_{i}\frac{\delta y^i}{F}$. Denoting $D^1:=v^1_iFg^{11}\hat{e}_{3}u^i_{1}$ one can rewrite the expression $F^2g^{ij}\frac{\partial^2 \phi}{\partial y^j\partial y^i}$ on $SM$ with respect to $\hat{e}_{3}$ as follows
\begin{equation*}%\label{second}
g^{11}\hat{e}_{3}\hat{e}_{3}\phi-D^1\hat{e}_{3}\phi=F^2g^{ij}\frac{\partial^2 \phi}{\partial y^j\partial y^i}.
\end{equation*}
Thus the principal order terms $g^{ij}\frac{\delta^{2}\phi}{\delta x^{i}\delta x^{j}}$ and $F^2g^{ij}\frac{\partial^{2}\phi}{\partial y^{i}\partial y^{j}}$ convert to $g^{ab}\hat{e}_{b}\hat{e}_{a}\phi-B^c\hat{e}_c\phi$
and $g^{11}\hat{e}_{3}\hat{e}_{3}\phi-D^1\hat{e}_{3}\phi$ on $SM$. On the other hand, the order of lower order terms in (\ref{lili}) do not change after rewriting them in terms of the basis $\{\hat{e}_1,\hat{e}_{2},\hat{e}_3\}$ on $SM$. Therefore (\ref{lili}) on $SM$ is written as
\begin{equation}\label{bibi}
\frac{\partial \phi}{\partial t}=g^{ab}\hat{e}_{b}\hat{e}_{a}\phi+g^{11}
\hat{e}_{3}\hat{e}_{3}\phi-B^c\hat{e}_c\phi-D^1\hat{e}_{3}\phi+\textrm{lower order terms},
\end{equation}
where, $a,b,c=1,2$. Using the fact that $g$ is positive definite, the coefficient
\begin{displaymath}
G=\left(\begin{array}{c|c}
g^{ab} & 0 \\
\hline
0 & g^{11}
\end{array}\right)_{3\times3},
\end{displaymath}
of principal order terms of (\ref{bibi}) is positive definite on $SM$. Therefore, by virtue of (\ref{po}) the  differential equation (\ref{bibi}) is a semi-linear strictly parabolic equation on $SM$.
\end{proof}
%%%%%%%%%%%%%%%%%%%%%%%%
\section{A vector field on $SM$}
Let $(M,F)$ and $(N,\bar{F})$ be two Finsler surfaces with the corresponding metric tensors $g$ and $h$, respectively. Let $(x^i,y^i)$ and $(\bar{x}^i,\bar{y}^i)$ be local coordinate systems on $TM$ and $TN$, respectively. Let $c$ be a geodesic on $M$. The natural lift of $c$ on $TM$, namely,
$$\tilde{c}:t\in I:\longrightarrow\tilde{c}(t)=(x^{i}(t),(\frac{dx^i}{dt})(t))\in TM,$$
  is a horizontal curve. That is, its tangent vector field $\dot{\tilde{c}}(t)=\frac{dx^i}{dt}\frac{\delta}{\delta x^i}$, is horizontal.
%  , see \cite[p.\ 58]{BM}.
   Consider a diffeomorphism
\begin{align*}
\varphi &:TM\longrightarrow TN,\\
&(x^i,y^i)\mapsto\varphi(x^i,y^i)=(\varphi^{\alpha}(x^i,y^i))=(\varphi^j(x^i,y^i),\varphi^{2+j}(x^i,y^i)),
\end{align*}
 such that $\bar{c}(t):=(\varphi\circ\tilde{c})(t)$ is a horizontal curve, where $i,j=1,2,$ and $\alpha =1,...,4$. Throughout this section, $\varphi$ takes horizontal curves to horizontal curves.

 Let us denote by $\Gamma^{i}_{\,jk}$ and $\bar{\Gamma}^{i}_{\,jk}$ the coefficients of horizontal covariant derivatives of Chern connection on $(M,F)$ and $(N,\bar{F})$, respectively. Then we have
\begin{align}\label{12}
\bar{\nabla}_{\dot{\bar{c}}}\dot{\bar{c}}&=\bar{\nabla}_{\dot{\bar{c}}}
\frac{d\bar{x}^j}{dt}\frac{\delta}{\delta \bar{x}^j}=\frac{d^{2}\bar{x}^{j}}{dt^{2}}\frac{\delta}{\delta \bar{x}^j}+\frac{d\bar{x}^j}{dt}\bar{\nabla}_{\dot{\bar{c}}}\frac{\delta}{\delta \bar{x}^j}
=(\frac{d^{2}\bar{x}^{j}}{dt^{2}}+\frac{d\bar{x}^j}{dt}\frac{d\bar{x}^k}{dt}\bar{\Gamma}^{i}_{\,jk})\frac{\delta}{\delta \bar{x}^{i}}.
\end{align}
On the other hand
\begin{align*}
\frac{d\bar{x}^i}{dt}=\frac{\delta\varphi^i}{\delta x^p}\frac{dx^p}{dt}, \quad \frac{d^{2}\bar{x}^i}{dt^2}=\frac{\delta^2\varphi^i}{\delta x^p\delta x^q}\frac{dx^p}{dt}\frac{dx^q}{dt}+\frac{\delta\varphi^i}{\delta x^p}\frac{d^2x^p}{dt^2}.
\end{align*}
Replacing the last equations in (\ref{12}) leads to
\begin{equation}\label{12+1+1}
\bar{\nabla}_{\dot{\bar{c}}}\dot{\bar{c}}=(\frac{\delta^2\varphi^i}{\delta x^p\delta x^q}\frac{dx^p}{dt}\frac{dx^q}{dt}+\frac{\delta\varphi^i}{\delta x^h}\frac{d^2x^h}{dt^2}+\bar{\Gamma}^{i}_{\,jk}\frac{\delta\varphi^j}{\delta x^p}\frac{\delta\varphi^k}{\delta x^q}\frac{dx^p}{dt}\frac{dx^q}{dt})\frac{\delta}{\delta \bar{x}^{i}},
\end{equation}
where, all the indices run over the range $1,2$.
The geodesic  $c$   on $M$, satisfies
\begin{equation}\label{12+1+1+1}
\frac{d^2x^h}{dt^2}+\Gamma^{h}_{\,pq}\frac{dx^p}{dt}\frac{dx^q}{dt}=0.
\end{equation}
Substituting   $ \frac{d^2x^h}{dt^2}$  from the last equation in (\ref{12+1+1}), leads to
\begin{equation*}
\bar{\nabla}_{\dot{\bar{c}}}\dot{\bar{c}}=\frac{dx^p}{dt}\frac{dx^q}{dt}(\frac{\delta^2\varphi^i}{\delta x^p\delta x^q}-\frac{\delta\varphi^i}{\delta x^h}\Gamma^{h}_{\,pq}+\bar{\Gamma}^{i}_{\,jk}\frac{\delta\varphi^j}{\delta x^p}\frac{\delta\varphi^k}{\delta x^q})\frac{\delta}{\delta \bar{x}^{i}}.
\end{equation*}
Next, let
\begin{eqnarray*}
A^{i}_{pq}:=\frac{\delta^2\varphi^i}{\delta x^p\delta x^q}-\frac{\delta\varphi^i}{\delta x^h}\Gamma^{h}_{\,pq}+\bar{\Gamma}^{i}_{\,jk}\frac{\delta\varphi^j}{\delta x^p}\frac{\delta\varphi^k}{\delta x^q}+F^2\frac{\partial^2\varphi^i}{\partial y^p\partial y^q}.
\end{eqnarray*}
Contracting $A^i_{pq}$ with $g^{pq}$ leads to the following operators
\begin{equation} \label{15}
(\Phi_{g,h}\varphi)^{i}:=g^{pq}(\frac{\delta^2\varphi^i}{\delta x^p\delta x^q}+F^2\frac{\partial^2\varphi^i}{\partial y^p\partial y^q}-\frac{\delta\varphi^i}{\delta x^h}\Gamma^{h}_{\,pq}+\bar{\Gamma}^{i}_{\,jk}\frac{\delta\varphi^j}{\delta x^p}\frac{\delta\varphi^k}{\delta x^q}),
\end{equation}
where, $(\Phi_{g,h}\varphi)^{i}=g^{pq}A^{i}_{pq}$ and $i=1,2$. For greater indices we consider the following operator.
\begin{equation} \label{15+1+1}
(\Phi_{g,h}\varphi)^{2+i}:=g^{pq}(\frac{\delta^2\varphi^{2+i}}{\delta x^p\delta x^q}+F^2\frac{\partial^2\varphi^{2+i}}{\partial y^p\partial y^q}+\frac{\partial\varphi^{2+i}}{\partial y^k}\frac{\delta N^k_q}{\delta x^p}),
\end{equation}
where, $i=1,2$. Summarizing the above definitions we have
\begin{equation*}
(\Phi_{g,h}\varphi)^{\alpha}=\left\{
\begin{array}{l}
(\Phi_{g,h}\varphi)^{i}\qquad \alpha=i\cr
(\Phi_{g,h}\varphi)^{2+i}\qquad \alpha=2+i,
\end{array}
\right.
\end{equation*}
where, $i=1,2.$ %In Riemannian case the operator (\ref{15}) reduces to the well known harmonic map Laplacian, see \cite[p.\ 85]{Chow1}.
Next, we show the operator $(\Phi_{g,h}\varphi)^{\alpha}$ is invariant under all diffeomorphisms on $TM$.
\begin{lem} \label{main2}
Let $(M,F)$ and $(N,\bar{F})$ be two Finsler surfaces with the corresponding metric tensors $g$ and $h$, respectively. If $\psi$ is a diffeomorphism from $TM$ to itself, then it leaves invariant the operator
$(\Phi_{g,h}\varphi)^\alpha$, that is
\begin{eqnarray*}
(\Phi_{\psi^{*}(g),h}\psi^{*}\varphi)^\alpha\mid_{(\tilde{x},\tilde{y})}=(\Phi_{g,h}\varphi)^\alpha\mid_{(x,y)},\qquad \alpha=1,..,4,
\end{eqnarray*}
where, $\tilde{x}^i=\psi^{*}x^i$ and $\tilde{y}^i=\psi^{*}y^i$.
\end{lem}
\begin{proof}
Let $(x^i,y^i)$ and $(\bar{x}^i,\bar{y}^i)$ be the two local coordinate systems on $TM$ and $TN$, respectively and $\tilde{x}^i=\psi^{*}x^i$ and $\tilde{y}^i=\psi^{*}y^i$. For $\alpha=i$, we have
\begin{align*}
(\Phi_{g,h}\varphi)^i\mid_{(x,y)}&=g^{pq}(x,y)\Big (\frac{\delta^{2}\varphi^{i}}{\delta x^{p} \delta x^{q}}(x,y)+F^2(x,y)\frac{\partial^2\varphi^i}{\partial y^p\partial y^q}(x,y)\\
&\quad-\frac{\delta \varphi^{i}}{\delta x^{k}}(x,y)\Gamma^{k}_{pq}(x,y)+\bar{\Gamma}^{i}_{jk}(\bar{x},\bar{y})\frac{\delta \varphi^{j}}{\delta x^{p}}(x,y)\frac{\delta \varphi^{k}}{\delta x^{q}}(x,y)\Big)\\
&=g^{pq}(\psi(\tilde{x},\tilde{y}))\Big(\frac{\delta^{2}\varphi^{i}}{\delta x^{p} \delta x^{q}}(\psi(\tilde{x},\tilde{y}))+F^2(\psi(\tilde{x},\tilde{y}))\frac{\partial^2\varphi^i}{\partial y^p\partial y^q}(\psi(\tilde{x},\tilde{y})) \\
&\quad-\frac{\delta \varphi^{i}}{\delta x^{k}}(\psi(\tilde{x},\tilde{y}))\Gamma^{k}_{pq}(\psi(\tilde{x},\tilde{y}))\\
&\quad+\bar{\Gamma}^{i}_{jk}(\bar{x},\bar{y})\frac{\delta \varphi^{j}}{\delta x^{p}}(\psi(\tilde{x},\tilde{y}))\frac{\delta \varphi^{k}}{\delta x^{q}}(\psi(\tilde{x},\tilde{y}))\Big)\\
&=(\psi^{*}
g)^{pq}(\tilde{x},\tilde{y})\Big(\frac{\delta^{2}(\psi^{*}\varphi)^{i}}{\delta \tilde{x}^{p} \delta \tilde{x}^{q}}(\tilde{x},\tilde{y})+(\psi^{*}F^2)(\tilde{x},\tilde{y})\frac{\partial^{2}(\psi^{*}\varphi)^{i}}{\partial \tilde{y}^{p} \partial \tilde{y}^{q}}(\tilde{x},\tilde{y})\\
&\quad-\frac{\delta (\psi^{*}\varphi )^{i}}{\delta \tilde{x}^{k}}(\tilde{x},\tilde{y})\Gamma(\psi^{*}g)^{k}_{pq}(\tilde{x},\tilde{y})\\
&\quad+\bar{\Gamma}^{i}_{jk}(\bar{x},\bar{y})\frac{\delta (\psi^{*}\varphi)^{j}}{\delta \tilde{x}^{p}}(\tilde{x},\tilde{y})\frac{\delta (\psi^{*}\varphi)^{k}}{\delta \tilde{x}^{q}}(\tilde{x},\tilde{y})\Big) \\
&=(\Phi_{\psi^{*}(g),h}\psi^{*}\varphi)^i\mid_{(\tilde{x},\tilde{y})}.\nonumber
\end{align*}
Similarly, for $\alpha=2+i$, one can show that
\begin{equation*}
(\Phi_{g,h}\varphi)^{2+i}\mid_{(x,y)}=(\Phi_{\psi^{*}(g),h}\psi^{*}\varphi)^{2+i}\mid_{(\tilde{x},\tilde{y})},
\end{equation*}
where, $i=1,2.$ This completes the proof.
\end{proof}
\begin{rem} \label{main3+1}
Let $(M,F)$ and $(N,\bar{F})$ be two Finsler surfaces with the corresponding metric tensors $g$ and $h$, respectively. Let $\varphi:TM\longrightarrow TN$, $\varphi(x^i,y^i)=(\varphi^\alpha(x^i,y^i))$, $\alpha=1,.., 4$, be a diffeomorphism and takes horizontal curves to horizontal curves. Given $\varphi_{0}:TM \longrightarrow TN$, we consider the following evolution equation
\begin{equation} \label{1666}
\frac{\partial}{\partial t}\varphi^\alpha=(\Phi_{g,h}\varphi)^\alpha,\hspace{0.6cm} \varphi_{(0)}=\varphi_{0}.
\end{equation}
By restricting $\varphi^\alpha$'s to $SM$ and using Lemma \ref{mm}, one can see that (\ref{1666}) is a strictly parabolic system. Hence, there is a unique solution for (\ref{1666}) in short time.
\end{rem}
\begin{cor} \label{main3}
Let $(M,\tilde{F})$ and $(N,\bar{F})$ be two Finsler surfaces with corresponding metric tensors $\tilde{g}$ and $h$, respectively. Let $N=M$ and $\varphi$ be the identity map $\varphi=Id:TM\longrightarrow TM$, $\varphi(x^i,y^i)=(x^i,y^i)$, then we have
\begin{equation*} %\label{17}
(\Phi_{\tilde{g},h}Id)^{\alpha}=(\Phi_{\tilde{g},h}Id)^{i}=\tilde{g}^{pq}(-\tilde{\Gamma}^{i}_{pq}+\bar{\Gamma}^{i}_{pq}),\quad \alpha=i,
\end{equation*}
where, $i=1,2$ and $\tilde{\Gamma}^{i}_{pq}, \bar{\Gamma}^{i}_{pq}$ are the coefficients of horizontal covariant derivatives of Chern connection with respect to the metrics $\tilde{g}$ and $h$, respectively.
\end{cor}
Let $\xi$ be a vector field on $TM$ with the components
\begin{align}\label{def;vect}
\xi:=(\Phi_{\tilde{g},h}Id)^{i}\frac{\partial}{\partial x^i}=\tilde{g}^{pq}(-\tilde{\Gamma}^{i}_{pq}+\bar{\Gamma}^{i}_{pq})\frac{\partial}{\partial x^i}.
\end{align}
 Using the fact that the difference of two connections is a tensor, $\xi$ is a globally well-defined vector field. It can be easily verified that the components of $\xi$ are homogeneous of degree zero on $y$, thus $\xi$ can be considered as a vector field on $SM$.
 %%%%%%%%%%%%%%%%%%%%%%%%%%%
\section{Ricci-DeTurck flow and its existence and uniqueness of solution}
There are several well known definitions for Ricci tensor in Finsler geometry. For instance, H. Akbar-Zadeh has considered two Ricci tensors on Finsler manifolds in his works namely, one is defined by $ Ric_{ij}:=[\frac{1}{2}F^{2}\mathcal{R}ic]_{y^{i}y^{j}}$ where, $\mathcal{R}ic$ is the \emph{Ricci scalar} defined by $\mathcal{R}ic:=g^{ik}R_{ik}=R^{i}_{\,\,i}$ and $R^{i}_{\,\,k}$ are defined by (\ref{18}), see \cite[p.\ 192]{BCS}. Another Ricci tensor is defined by
 $Rc_{ij}:=\frac{1}{2} (\textsf{R}_{ij}+\textsf{R}_{ji})$, where $ \textsf{R}_{ij}$ is the trace of $hh$-curvature defined by $ \textsf{R}_{ij}=R^{\,\,l}_{i\,\,lj}$. The difference between these two Ricci tensors is the additional term $\frac{1}{2}y^k\frac{\partial  \textsf{R}_{jk}}{\partial y^i}$ appeared in the first definition. More precisely, we have $Ric_{ij}-Rc_{ij}=\frac{1}{2}y^k\frac{\partial  \textsf{R}_{jk}}{\partial y^i}$. D. Bao based on the first definition of Ricci tensor has considered the following Ricci flow in Finsler geometry,
\begin{equation}\label{20002}
\frac{\partial}{\partial t}g_{jk}(t)=-2Ric_{jk},\hspace{0.6cm}g_{(t=0)}=g_{0},
\end{equation}
where, $g_{jk}(t)$ is a family of Finslerian metrics defined on $\pi^{*}TM\times[0,T)$.
Contracting (\ref{20002}) with $y^{j}y^{k}$, via Euler's theorem, leads to $\frac{\partial}{\partial t}F^{2}=-2F^{2}\mathcal{R}ic$. That is,
\begin{equation} \label{20}
\frac{\partial}{\partial t}\log F(t)=-\mathcal{R}ic,\hspace{0.6cm}F(t=0):=F_{0},
\end{equation}
where, $F_{0}$ is the initial Finsler structure, see \cite{Bao}. Here and everywhere in the present work  we consider the first Akbar-Zadeh's definition of Ricci tensor and the related  Ricci flow (\ref{20}). One of the advantages of the Ricci quantity $Ric_{ij}$, used in the present work is its independence on the choice of Cartan, Berwald or Chern connections.
\begin{defn} \label{Main4}
Let $M$ be a compact surface with a fixed background Finsler structure $\bar{F}$ and related Finsler metric $h$. Assume that for all $t\in [0,T)$, $\tilde{F}(t)$ is a one-parameter family of Finsler structures on $TM$ and $\tilde{g}(t)$ is the tensor metric related to $\tilde{F}(t)$. We say that $\tilde{F}(t)$ is a solution to the Finslerian Ricci-DeTurck flow if
\begin{equation} \label{22}
\frac{\partial}{\partial t}\tilde{F}^{2}(t)=-2\tilde{F}^{2}(t)\mathcal{R}ic(\tilde{g}(t))-\mathcal{L}_{\xi}\tilde{F}^{2}(t),
\end{equation}
where, $\mathcal{L}_{\xi}$ is the Lie derivative with respect to the vector field $\xi=(\Phi_{\tilde{g}(t),h}Id)^{i}\frac{\partial}{\partial x^i}$ on $SM$ as mentioned earlier.
\end{defn}
The following theorem shows that the Ricci-DeTurck flow (\ref{22}) is well defined and has a unique solution on a short time interval.\\

{\bf Proof of Theorem \ref{main8}.}
Let $M$ be a compact surface with a fixed background Finsler structure $\bar{F}$ and the related Finsler metric $h$. Here, all the indices run over the range $1,2$. The Ricci-DeTurck flow (\ref{22}) can be written in the following form
\begin{equation}\label{28}
y^{p}y^{q}\frac{\partial}{\partial t}\tilde{g}_{pq}(t)=-2\tilde{F}^{2}(t)\mathcal{R}ic(\tilde{g}(t))-\mathcal{L}_{\xi}y^{p}y^{q}\tilde{g}_{pq}(t),
\end{equation}
where, $\tilde{g}(t)$ is the metric tensor related to $\tilde{F}(t)$. Also we have
\begin{equation*}
\mathcal{L}_{\xi}(y^{p}y^{q}\tilde{g}_{pq})=y^{p}y^{q}\mathcal{L}
_{\xi}\tilde{g}_{pq}+2y^{p}\tilde{g}_{pq}\mathcal{L}_{\xi}y^{q}.
\end{equation*}
Therefore, (\ref{28}) becomes
\begin{equation}\label{2800}
y^{p}y^{q}\frac{\partial}{\partial t}\tilde{g}_{pq}(t)=-2\tilde{F}^{2}(t)\mathcal{R}ic(\tilde{g}(t))-y^{p}y^{q}\mathcal{L}
_{\xi}\tilde{g}_{pq}-2y^{p}\tilde{g}_{pq}\mathcal{L}_{\xi}y^{q}.
\end{equation}
By means of the Lie derivative formula (\ref{FINAL}) along $\xi$ we have
\begin{equation}\label{30}
y^py^q\mathcal{L}_{\xi}\tilde{g}_{pq}=2y^py^q\nabla_{p}\xi_{q},
\end{equation}
where, $\nabla_{p}$ is the horizontal covariant derivative in Chern connection. Using its $h$-metric compatibility, $\nabla_{p}\xi_{q}$ becomes
\begin{eqnarray*}
\nabla_{p}\xi_{q}=(\nabla_{p}\tilde{g}_{ql}\xi^{l})=
\tilde{g}_{ql}(\nabla_{p}\xi^{l}).
\end{eqnarray*}
As mentioned earlier, if we denote the coefficients of horizontal covariant derivatives of Chern connection with respect to the metric tensors $h$ and $\tilde{g}$ by $\Gamma(h)$ and $\Gamma(\tilde{g})$, respectively, then by definition \eqref{def;vect} of $\xi$ we have
\begin{align*}
\nabla_{p}\xi_{q}&=\tilde{g}_{ql}(\delta_{p}\xi^{l}+\Gamma(\tilde{g})^{l}_{pw}\xi^{w})\\
&=\tilde{g}_{ql}[\delta_{p}(\tilde{g}^{mn}(\Gamma(h)^{l}_{mn}-\Gamma(\tilde{g})^{l}_{mn}))]+\tilde{g}_{ql}\Gamma(\tilde{g})^{l}_{pw}\xi^{w}\\
&=\tilde{g}_{ql}[(\delta_{p}\tilde{g}^{mn})(\Gamma(h)^{l}_{mn}-\Gamma(\tilde{g})^{l}_{mn})+\tilde{g}^{mn}\delta_{p}(\Gamma(h)^{l}_{mn})-\tilde{g}^{mn}\delta_{p}(\Gamma(\tilde{g})^{l}_{mn})]\nonumber\\
&\quad +\tilde{g}_{ql}\Gamma(\tilde{g})^{l}_{pw}\xi^{w}\\
&=\frac{1}{2}\tilde{g}^{mn}(\delta_{p}\delta_{q}\tilde{g}_{mn}
-\delta_{p}\delta_{n}\tilde{g}_{qm}-\delta_{p}\delta_{m}\tilde{g}_{qn})\nonumber\\
&\quad -\frac{1}{2}\tilde{g}_{ql}\tilde{g}^{mn}(\delta_{p}\tilde{g}^{ls})
(\delta_{n}\tilde{g}_{sm}-\delta_{s}\tilde{g}_{mn}+\delta_{m}\tilde{g}_{ns})\\
&\quad +\tilde{g}_{ql}(\delta_{p}\tilde{g}^{mn})(\Gamma(h)^{l}_{mn}-\Gamma(\tilde{g})^{l}_{mn})
 +\tilde{g}_{ql}\tilde{g}^{mn}\delta_{p}(\Gamma(h)^{l}_{mn})+\tilde{g}_{ql}\Gamma(\tilde{g})^{l}_{pw}\xi^{w}.\\
\end{align*}
Using the last equation, (\ref{30}) is written
\begin{align} \label{31}
y^{p}y^{q}\mathcal{L}_{\xi}\tilde{g}_{pq}=&y^py^q\tilde{g}^{mn}(\delta_{p}\delta_{q}\tilde{g}_{mn}
-\delta_{p}\delta_{n}\tilde{g}_{qm}-\delta_{p}\delta_{m}\tilde{g}_{qn})\nonumber\\
&-y^py^q\tilde{g}_{ql}\tilde{g}^{mn}(\delta_{p}\tilde{g}^{ls})
(\delta_{n}\tilde{g}_{sm}-\delta_{s}\tilde{g}_{mn}+\delta_{m}\tilde{g}_{ns})\nonumber\\
&+2y^py^q\tilde{g}_{ql}(\delta_{p}\tilde{g}^{mn})(\Gamma(h)^{l}_{mn}-\Gamma(\tilde{g})^{l}_{mn})\nonumber\\
&+2y^py^q\tilde{g}_{ql}\tilde{g}^{mn}\delta_{p}(\Gamma(h)^{l}_{mn})+2y^py^q\tilde{g}_{ql}\Gamma
(\tilde{g})^{l}_{pw}\xi^{w}.
\end{align}
Also we have
\begin{equation} \label{32}
-2\tilde{F}^{2}\mathcal{R}ic(\tilde{g})=-2\tilde{F}^{2}R^{n}_{\,\,n}
=-2\tilde{F}^{2}l^{q}R^{\,\,n}_{q\,\,np}l^{p},
\end{equation}
where, $R^{\,\,n}_{q\,\,np}$ are the components of hh-curvature tensor of Chern connection and $l^{q}=\frac{y^{q}}{\tilde{F}}$ are the components of Liouville vector field. Replacing (\ref{77}) in (\ref{32}) and using the definition of $\Gamma(\tilde{g})$, yields
\begin{align*}% \label{34+1}
-2\tilde{F}^{2}\mathcal{R}ic(\tilde{g})&=-2\tilde{F}^{2}l^{q}R^{\,\,n}_{q\,\,np}l
^{p}\nonumber\\
&=-2y^{p}y^{q}(\delta_{n}\Gamma^{n}_{qp}(\tilde{g})-\delta_{p}\Gamma^{n}_{qn}(\tilde{g})
+\Gamma^{n}_{mn}(\tilde{g})\Gamma^{m}_{qp}(\tilde{g})-\Gamma^{n}_{mp}(\tilde{g})\Gamma^{m}
_{qn}(\tilde{g}))\nonumber \\
&=-2y^{p}y^{q}[\delta_{n}(\frac{1}{2}\tilde{g}^{mn}(\delta_{q}\tilde{g}_{mp}+\delta_{p}\tilde{g}_{mq}-\delta_{m}\tilde{g}_{pq}))]\nonumber\\
&\quad +2y^{p}y^{q}[\delta_{p}(\frac{1}{2}\tilde{g}^{mn}(\delta_{q}\tilde{g}_{mn}+\delta_{n}\tilde{g}_{qm}-\delta_{m}\tilde{g}_{qn}))] \nonumber\\
&\quad -2y^{p}y^{q}(\Gamma^{n}_{mn}(\tilde{g})\Gamma^{m}_{qp}(\tilde{g})-\Gamma^{n}_{mp}(\tilde{g})\Gamma^
{m}_{qn}(\tilde{g})).
\end{align*}
By applying the $\delta_{p}$ derivative we have
\begin{align} \label{34}
-2\tilde{F}^{2}\mathcal{R}ic(\tilde{g})=&y^{p}y^{q}\tilde{g}^{mn}(\delta_{p}\delta_{q}\tilde{g}_{mn}+\delta_{n}\delta_{m}
\tilde{g}_{pq}-\delta_{p}\delta_{m}\tilde{g}_{qn}-\delta_{n}\delta_{q}
\tilde{g}_{mp})\nonumber\\
&-y^{p}y^{q}(\delta_{n}\tilde{g}^{nm})(\delta_{q}\tilde{g}_{mp}+\delta_{p}\tilde{g}_{qm}
-\delta_{m}\tilde{g}_{pq})\nonumber\\
&+y^{p}y^{q}(\delta_{p}\tilde{g}^{mn})(\delta_{q}\tilde{g}_{mn}+\delta_{n}\tilde{g}_{qm}
-\delta_{m}\tilde{g}_{qn})\nonumber\\
&-2y^{p}y^{q}(\Gamma^{n}_{mn}(\tilde{g})\Gamma^{m}_{qp}(\tilde{g})-\Gamma^{n}_{mp}(\tilde{g})
\Gamma^{m}_{qn}(\tilde{g})).
\end{align}
Substituting (\ref{31}) and (\ref{34}) in (\ref{2800}), we obtain
\begin{align} \label{36}
y^{p}y^{q}\frac{\partial}{\partial t}\tilde{g}_{pq}(t)=&y^py^q\tilde{g}^{mn}\delta_{n}\delta_{m}\tilde{g}_{pq}\\
&-y^{p}y^{q}(\delta_{n}\tilde{g}^{nm})(\delta_{q}\tilde{g}_{mp}+\delta_{p}\tilde{g}_{qm}
-\delta_{m}\tilde{g}_{pq})\nonumber\\
&+y^{p}y^{q}(\delta_{p}\tilde{g}^{mn})(\delta_{q}\tilde{g}_{mn}+\delta_{n}\tilde{g}_{qm}
-\delta_{m}\tilde{g}_{qn})\nonumber\\
&-2y^{p}y^{q}(\Gamma^{n}_{mn}(\tilde{g})\Gamma^{m}_{qp}(\tilde{g})-\Gamma^{n}_{mp}(\tilde{g})
\Gamma^{m}_{qn}(\tilde{g}))\nonumber\\
&+y^py^q\tilde{g}_{ql}\tilde{g}^{mn}(\delta_{p}\tilde{g}^{ls})
(\delta_{n}\tilde{g}_{sm}-\delta_{s}\tilde{g}_{mn}+\delta_{m}\tilde{g}_{ns})\nonumber\\
&-2y^py^q\tilde{g}_{ql}(\delta_{p}\tilde{g}^{mn})(\Gamma(h)^{l}_{mn}-\Gamma(\tilde{g})^{l}_{mn})\nonumber\\
&-2y^py^q\tilde{g}_{ql}\tilde{g}^{mn}\delta_{p}(\Gamma(h)^{l}_{mn})-2y^py^q\tilde{g}_{ql}\Gamma
(\tilde{g})^{l}_{pw}\xi^{w}\nonumber\\
&-2y^{p}\tilde{g}_{pq}\mathcal{L}_{\xi}y^{q}.\nonumber
\end{align}
Using Euler's theorem yields
\begin{equation} \label{RE3}
y^{p}y^{q}\frac{\partial^{2}\tilde{g}_{pq}}{\partial y^{n}\partial y^{m}}=\frac{\partial ^{2}}{\partial y^{n}\partial y^{m}}(y^{p}y^{q}\tilde{g}_{pq})-2\tilde{g}_{nm}=0.
\end{equation}
In order to get a strictly parabolic system, by virtue of (\ref{RE3}) we add the zero term $\tilde{F}^2 y^{p}y^{q}\tilde{g}^{mn}\frac{\partial^{2}\tilde{g}_{pq}}{\partial y^{n}\partial y^{m}}=0$  to the right hand side of (\ref{36}). Therefore, we have
\begin{align} \label{36+11}
y^{p}y^{q}\Big(&\frac{\partial}{\partial t}\tilde{g}_{pq}(t)-\tilde{g}^{mn}\delta_{n}\delta_{m}\tilde{g}_{pq}-\tilde{F}^2 \tilde{g}^{mn}\frac{\partial^{2}\tilde{g}_{pq}}{\partial y^{n}\partial y^{m}}\\
&+(\delta_{n}\tilde{g}^{nm})(\delta_{q}\tilde{g}_{mp}+\delta_{p}\tilde{g}_{qm}
-\delta_{m}\tilde{g}_{pq})\nonumber\\
&-(\delta_{p}\tilde{g}^{mn})(\delta_{q}\tilde{g}_{mn}+\delta_{n}\tilde{g}_{qm}
-\delta_{m}\tilde{g}_{qn})\nonumber\\
&+2(\Gamma^{n}_{mn}(\tilde{g})\Gamma^{m}_{qp}(\tilde{g})-\Gamma^{n}_{mp}(\tilde{g})
\Gamma^{m}_{qn}(\tilde{g}))\nonumber\\
&-\tilde{g}_{ql}\tilde{g}^{mn}(\delta_{p}\tilde{g}^{ls})
(\delta_{n}\tilde{g}_{sm}-\delta_{s}\tilde{g}_{mn}+\delta_{m}\tilde{g}_{ns})\nonumber\\
&+2\tilde{g}_{ql}(\delta_{p}\tilde{g}^{mn})(\Gamma(h)^{l}_{mn}-\Gamma(\tilde{g})^{l}_{mn})\nonumber\\
&+2\tilde{g}_{ql}\tilde{g}^{mn}\delta_{p}(\Gamma(h)^{l}_{mn})-2y^py^q\tilde{g}_{ql}\Gamma
(\tilde{g})^{l}_{pw}\xi^{w}+2\frac{l_{q}}{F}\tilde{g}_{np}\mathcal{L}_{\xi}y^{n}\Big)=0.\nonumber
\end{align}
On the other hand, applying  twice the vector field  $\frac{\delta}{\delta x^{n}}$ on the components of metric tensor  $\tilde{g}_{pq}$ yields
\begin{align*} \label{RE1}
\delta_{n}\delta_{m}\tilde{g}_{pq}=&\frac{\partial^{2}\tilde{g}_{pq}}{\partial x^{n}\partial x^{m}}-\frac{\partial N^{k}_{m}}{\partial x^{n}}\frac{\partial \tilde{g}_{pq}}{\partial y^{k}}-N^{k}_{m}\frac{\partial^{2}\tilde{g}_{pq}}{\partial x^{n}\partial y^{k}}-N^{l}_{n}\frac{\partial^{2}\tilde{g}_{pq}}{\partial y^{l}\partial x^{m}}\nonumber\\
&+N^{l}_{n}\frac{\partial N_{m}^{k}}{\partial y^{l}}\frac{\tilde{g}_{pq}}{\partial y^{k}}+
N^{l}_{n}N_{m}^k\frac{\partial^2\tilde{g}_{pq}}{\partial y^k\partial y^l}.\nonumber
\end{align*}
Convecting the last equation with $y^{p}y^{q}$ and using (\ref{RE3}) we have
\begin{equation*} %\label{RE2}
y^{p}y^{q}\delta_{n}\delta_{m}\tilde{g}_{pq}
=y^{p}y^{q}\frac{\partial^{2}\tilde{g}_{pq}}{\partial x^{n}\partial x^{m}}.
\end{equation*}
Remark that, in the term $y^py^q\tilde{g}^{mn}\delta_{n}\delta_{m}\tilde{g}_{pq}$ in (\ref{36}), there is no term containing derivatives of $\tilde{g}$ except $y^{p}y^{q}\frac{\partial^{2}\tilde{g}_{pq}}{\partial x^{n}\partial x^{m}}$. One can rewrite (\ref{36+11}) as follows
\begin{align} \label{36+111}
y^{p}y^{q}\Big(\frac{\partial}{\partial t}\tilde{g}_{pq}(t)-\tilde{g}^{mn}\delta_{n}\delta_{m}\tilde{g}_{pq}-\tilde{F}^2 \tilde{g}^{mn}\frac{\partial^{2}\tilde{g}_{pq}}{\partial y^{n}\partial y^{m}}
+\textrm{lower order terms}\Big)=0.
\end{align}
Recall that, $M$ is a 2-dimensional Finsler surface and hence is isotropic. Thus, $R^{\,\,n}_{q\,\,np}$ can be part of a symmetric quadratic form, namely, it is symmetric with respect to the indices $p$ and $q$, see \cite[p.\ 152]{HAZ}. Therefore, by means of symmetry of $\mathcal{L}_{\xi}(y^{p}y^{q}\tilde{g}_{pq})$ with respect to the indices $p$ and $q$ we conclude that (\ref{36+111}) is symmetric with respect to the indices $p$ and $q$. If a symmetric bilinear form vanishes on the diagonal, then by the polarization identity it vanishes identically. Therefore, from (\ref{36+111}) we have
\begin{equation} \label{36+1111}
\frac{\partial}{\partial t}\tilde{g}_{pq}(t)-\tilde{g}^{mn}\delta_{n}\delta_{m}\tilde{g}_{pq}-\tilde{F}^2 \tilde{g}^{mn}\frac{\partial^{2}\tilde{g}_{pq}}{\partial y^{n}\partial y^{m}}
+\textrm{lower order terms}=0.
\end{equation}
By restricting the metric tensor $\tilde{g}$ on $p^{*}TM$ and using Lemma \ref{mm} we can rewrite (\ref{36+1111}) in terms of the basis $\{\hat{e}_1,\hat{e}_{2},\hat{e}_3\}$ on $SM$ as follows
\begin{equation}\label{asli}
\frac{\partial}{\partial t}\tilde{g}_{pq}=\tilde{g}^{ab}\hat{e}_{b}\hat{e}_{a}\tilde{g}_{pq}+\tilde{g}^{11}
\hat{e}_{3}\hat{e}_{3}\tilde{g}_{pq}-B^c\hat{e}_c\tilde{g}_{pq}-D^1\hat{e}_{3}\tilde{g}_{pq}+\textrm{lower order terms},
\end{equation}
where, $B^c:=v^c_i\tilde{g}^{ab}\hat{e}_{b}(u^i_a)$ and $D^1:=v^1_i\tilde{F}\tilde{g}^{11}\hat{e}_{3}u^i_{1}$ as mentioned in Lemma \ref{mm} and all the indices in (\ref{asli}) run over the range $1,2$.

By assumption $M$ is compact and the sphere bundle $SM$ as well. Also, the metric tensor $\tilde{g}^{mn}$ remains positive definite along the Ricci flow, see \cite{YB2}, Corollary 3.7. Since the coefficients of principal (second) order terms of (\ref{asli}) are  positive definite, by Definition \ref{semipar}, it is a semi-linear strictly parabolic system on $SM$. Therefore, the standard existence and uniqueness theorem for parabolic systems on compact domains implies that, (\ref{asli}) has a unique solution on $SM$. Equation (\ref{asli}) is a special case of the general flow (\ref{AR}) and $\tilde{g}(t)$ is a solution to it. Therefore, by means of Lemma \ref{RE2}, $\tilde{g}(t)$ satisfies the integrability condition or equivalently, there exists a Finsler structure $\tilde{F}(t)$ on $TM$ such that $\tilde{g}_{ij}=\frac{1}{2}\frac{\partial^2 \tilde{F}}{\partial y^i\partial y^j}$.
Hence, $\tilde{g}$ is a Finsler metric and determines a Finsler structure  $\tilde{F}^{2}:=\tilde{g}_{pq}y^{p}y^{p}$ which is a unique solution to the Finsler Ricci-DeTurck flow. This completes the proof of Theorem \ref{main8}.\hspace{\stretch{1}}$\Box$

%%%%%%%%%%%%%%%%%%%%%%%%%%%%%%%
\section{Short time solution to the Ricci flow on Finsler surfaces}
In this section, we will show that there is a one-to-one correspondence between the solutions to Ricci flow and Ricci-DeTurck flow on Finsler surfaces. Here, we recall some results which will be used in the sequel.
\begin{Lemma} \label{main9}
\cite[p.\ 82]{Chow1} Let $\{X_{t}:0\leq t<T\leq \infty\}$ be a continuous time-dependent family of vector fields on a compact manifold $M$, then there exists a one-parameter family of diffeomorphisms $\{\varphi_{t}:M \longrightarrow M;\quad 0 \leq t<T \leq \infty\}$ defined on the same time interval such that
\begin{equation*}% \label{lem}
\left\{
\begin{array}{l}
\frac{\partial}{\partial t} \varphi_{t}(x)=X_{t}[\varphi_{t}(x)],\cr
\varphi_{0}(x)=x,
\end{array}
\right.
\end{equation*}
for all $x\in M$ and $t\in[0,T)$.
\end{Lemma}
\begin{rem}\label{remark1}
Let $M$ be a compact Finsler surface. According to Lemma \ref{main9} there exists a unique one-parameter family of diffeomorphisms $\tilde{\varphi}_{t}$ on $SM$, such that
\begin{equation*}% \label{45}
\left\{
\begin{array}{l}
\frac{\partial}{\partial t}\varphi_{t}(z)=\xi(\varphi_{t}(z),t),\cr
\varphi_{0}=Id_{SM},
\end{array}
\right.
\end{equation*}
where, $z=(x,[y])\in SM$ and $t\in[0,T)$.
\end{rem}
\begin{rem}\label{remark2}
Let $\tilde{g}_{pq}$ be a solution to the Ricci-DeTurck flow and $\varphi_{t}$ the one-parameter global group of diffeomorphisms according to the vector field $\xi$. Since $\xi$ is a vector field on $SM$, then  $\varphi_{t}$ are homogeneous of degree zero. Zero-homogeneity of $\tilde{g}_{pq}$ implies that $\varphi_{t}^{*}(\tilde{g}_{pq})$ be also homogeneous of degree zero. In fact,
\begin{equation*}
(\varphi_{t}^{*}\tilde{g}_{pq})(x,\lambda y)=\tilde{g}_{pq}(\varphi_{t}(x,\lambda y))=\tilde{g}_{pq}(\varphi_{t}(x,y))=
(\varphi_{t}^{*}\tilde{g}_{pq})(x,y).
\end{equation*}
Using the fact that $\tilde{g}_{pq}$ is positive definite and $\varphi_{t}^{*}$ are diffeomorphisms, $\varphi_{t}^{*}(\tilde{g}_{pq})$ is also positive definite. As well $\varphi_{t}^{*}(\tilde{g}_{pq})$ is symmetric. More intuitively,
\begin{equation*}
(\varphi_{t}^{*}\tilde{g})(X,Y)=g(\varphi_{t_{_*}}(X),\varphi_{t_{_*}}(Y))=g(\varphi_{t_{_*}}(Y),\varphi_{t_{_*}}(X))=(\varphi_{t}^{*}\tilde{g})(Y,X).
\end{equation*}
Therefore, $\varphi_{t}^{*}(\tilde{g}_{pq})$ determines a Finsler structure as follows
\begin{equation*}
F^{2}:=g_{pq}\tilde{y}^{p}\tilde{y}^{q},
\end{equation*}
where, $g_{pq}:=\varphi_{t}^{*}(\tilde{g}_{pq})$ and $\varphi_{t}^{*}y^p:=\tilde{y}^{p}$.
\end{rem}
\begin{lem}
Let $\varphi_{t}$ be a global one parameter group of diffeomorphisms corresponding to the vector field $\xi$ and $(\gamma^{i}_{jk})_{\tilde{g}}$ and $(G^{i})_{\tilde{g}}$ are the Christoffel symbols and spray coefficients related to the Finsler metric $\tilde{g}$, respectively. Then we have
\begin{align}
&\varphi_{t}^{*}((\gamma_{jk}^{i})_{\tilde{g}})=(\gamma_{jk}^{i})_{\varphi_{t}^{*}(\tilde{g})},\label{Christoffel}\\
&\varphi_{t}^{*}(G^{i}_{\tilde{g}})=G^{i}_{\varphi_{t}^{*}(\tilde{g})},\label{spray}
\end{align}
where, $(\gamma_{jk}^{i})_{\tilde{g}}=\tilde{g}^{is}\frac{1}{2}(\frac{\partial \tilde{g}_{sj}}{\partial x^k}-\frac{\partial \tilde{g}_{jk}}{\partial x^s}+\frac{\partial \tilde{g}_{ks}}{\partial x^j})$ and $G^{i}_{\tilde{g}}=\frac{1}{2}(\gamma^{i}_{jk})_{\tilde{g}}y^jy^k$.
\end{lem}
\begin{proof}
Let us denote $\varphi_{t}^{*}x^i=\tilde{x}^i$ and $\varphi_{t}^{*}y^i=\tilde{y}^i$. By definition, we have
\begin{align*}
\varphi_{t}^{*}((\gamma^{i}_{jk})_{\tilde{g}})&=\varphi_{t}^{*}(\tilde{g}^{is}\frac{1}{2}(\frac{\partial \tilde{g}_{sj}}{\partial x^k}-\frac{\partial \tilde{g}_{jk}}{\partial x^s}+\frac{\partial \tilde{g}_{ks}}{\partial x^j}))\\
&=\varphi_{t}^{*}(\tilde{g}^{is})\varphi_{t}^{*}(\frac{1}{2}(\frac{\partial \tilde{g}_{sj}}{\partial x^k}-\frac{\partial \tilde{g}_{jk}}{\partial x^s}+\frac{\partial \tilde{g}_{ks}}{\partial x^j}))\\
&=\varphi_{t}^{*}(\tilde{g}^{is})\frac{1}{2}(\frac{\partial\varphi_{t}^{*}( \tilde{g}_{sj})}{\partial \tilde{x}^k}-\frac{\partial\varphi_{t}^{*}( \tilde{g}_{jk})}{\partial \tilde{x}^s}+\frac{\partial\varphi_{t}^{*}(\tilde{g}_{ks})}{\partial \tilde{x}^j})\\
&=(\gamma_{jk}^{i})_{\varphi_{t}^{*}(\tilde{g})}.
\end{align*}
Next, by means of (\ref{Christoffel}) we have
\begin{align*}
\varphi_{t}^{*}(G^{i}_{\tilde{g}})&=\varphi_{t}^{*}(\frac{1}{2}(\gamma^{i}_{jk})_{\tilde{g}}y^jy^k)=\frac{1}{2}\varphi_{t}^{*}((\gamma^{i}_{jk})_{\tilde{g}})\varphi_{t}^{*}y^j\varphi_{t}^{*}y^k\\
&=\frac{1}{2}(\gamma^{i}_{jk})_{\varphi_{t}^{*}(\tilde{g})}\tilde{y}^{j}\tilde{y}^{k}=G^{i}_{\varphi_{t}^{*}(\tilde{g})}.
\end{align*}
This completes the proof.
\end{proof}
\begin{lem}\label{lemmohem}
Let $\varphi_{t}$ be a global one parameter group of diffeomorphisms generating the vector field $\xi$ and $\mathcal{R}ic_{\tilde{g}}$ the Ricci scalar related to the Finsler metric $\tilde{g}$, then we have
\begin{equation*}
\varphi_{t}^{*}(\mathcal{R}ic_{\tilde{g}})=\mathcal{R}ic_{\varphi_{t}^{*}(\tilde{g})}.
\end{equation*}
\end{lem}
\begin{proof}
Let us consider the \emph{reduced hh-curvature tensor} $R^{i}_{\,\,k}$ which is expressed entirely in terms of the $x$ and $y$ derivatives of the spray coefficients $G^{i}_{\tilde{g}}$.
\begin{equation*}
(R^{i}_{\,\,k})_{\tilde{g}}:=\frac{1}{\tilde{F}^2}(2\frac{\partial G^{i}_{\tilde{g}}}{\partial x^{k}}-\frac{\partial^{2}G^{i}_{\tilde{g}}}{\partial x^j\partial y^k}y^{j}+2G^{j}_{\tilde{g}}\frac{\partial^{2}G^{i}_{\tilde{g}}}{\partial y^j\partial y^k}-\frac{\partial G^{i}_{\tilde{g}}}{\partial y^{j}}\frac{\partial G^{j}_{\tilde{g}}}{\partial y^{k}}).
\end{equation*}
Therefore, we have
\begin{align*}
\varphi_{t}^{*}((R^{i}_{\,\,k})_{\tilde{g}})&=\varphi_{t}^{*}
(\frac{1}{\tilde{F}^2}(2\frac{\partial G^{i}_{\tilde{g}}}{\partial x^{k}}-\frac{\partial^{2}G^{i}_{\tilde{g}}}{\partial x^j\partial y^k}y^{j}+2G^{j}_{\tilde{g}}\frac{\partial^{2}G^{i}_{\tilde{g}}}{\partial y^j\partial y^k}-\frac{\partial G^{i}_{\tilde{g}}}{\partial y^{j}}\frac{\partial G^{j}_{\tilde{g}}}{\partial y^{k}}))\nonumber\\
&=\varphi_{t}^{*}(\frac{1}{\tilde{F}^2})\varphi_{t}^{*}(2\frac{\partial G^{i}_{\tilde{g}}}{\partial x^{k}}-\frac{\partial^{2}G^{i}_{\tilde{g}}}{\partial x^j\partial y^k}y^{j}+2G^{j}_{\tilde{g}}\frac{\partial^{2}G^{i}_{\tilde{g}}}{\partial y^j\partial y^k}-\frac{\partial G^{i}_{\tilde{g}}}{\partial y^{j}}\frac{\partial G^{j}_{\tilde{g}}}{\partial y^{k}}).
\end{align*}
Thus, we get
\begin{align*}
\varphi_{t}^{*}((R^{i}_{\,\,k})_{\tilde{g}})=&\frac{1}{\varphi_{t}^{*}(\tilde{F}^2)}(2\frac{\partial(\varphi_{t}^{*} (G^{i}_{\tilde{g}}))}{\partial \tilde{x}^{k}}-\frac{\partial^{2}(\varphi_{t}^{*}(G^{i}_{\tilde{g}}))}{\partial \tilde{x}^j\partial \tilde{y}^k}\tilde{y}^{j}\nonumber\\
&+2\varphi_{t}^{*}(G^{j}_{\tilde{g}})\frac{\partial^{2}(\varphi_{t}^{*}(G^{i}_{\tilde{g}}))}{\partial \tilde{y}^j\partial \tilde{y}^k}-\frac{\partial(\varphi_{t}^{*}( G^{i}_{\tilde{g}}))}{\partial \tilde{y}^{j}}\frac{\partial (\varphi_{t}^{*}(G^{j}_{\tilde{g}}))}{\partial \tilde{y}^{k}}).
\end{align*}
Putting $i=k$ in this equation together with (\ref{spray}) implies
\begin{equation*}
\varphi_{t}^{*}(\mathcal{R}ic_{\tilde{g}})=\mathcal{R}ic_{\varphi_{t}^{*}(\tilde{g})},
\end{equation*}
as we have claimed.
\end{proof}
Now we are in a position to prove the following proposition.
\begin{prop} \label{main11}
Fix a compact Finsler surface $(M,\bar{F})$ with related Finsler metric tensor $h$. Let $\tilde{F}(t)$ be a family of solutions to the Ricci-DeTurck flow
\begin{eqnarray} \label{46}
\frac{\partial}{\partial t}\tilde{F}^{2}(t)=-2\tilde{F}^{2}(t)\mathcal{R}ic(\tilde{g}(t))-\mathcal{L}_{\xi}\tilde{F}^{2}(t),
\end{eqnarray}
where, $\xi=(\Phi_{\tilde{g}(t),h}Id)^{i}\frac{\partial}{\partial x^i}$ and $t\in[0,T)$. Moreover, let $\varphi_{t}$ be a one-parameter family of diffeomorphisms satisfying
\begin{eqnarray*} %\label{47}
\frac{\partial}{\partial t}\varphi_{t}(z)=\xi(\varphi_{t}(z),t),\nonumber
\end{eqnarray*}
for $z\in SM$ and $t\in[0,T)$. Then the Finsler structures $F(t)$ form a solution to the Finslerian Ricci flow (\ref{20}) where, $F(t)$ is defined by
\begin{eqnarray*} %\label{48}
F^{2}(t):=g_{pq}\tilde{y}^{p}\tilde{y}^{q}=\varphi_{t}^{*}(\tilde{F}^{2}(t)).\nonumber
\end{eqnarray*}
where, $g_{pq}:=\varphi_{t}^{*}(\tilde{g}_{pq})$ and $\varphi_{t}^{*}y^p:=\tilde{y}^{p}$.
\end{prop}
\begin{proof}
In order to show $F(t)$ form a solution to the Finslerian Ricci flow (\ref{20}) we have to show $\frac{\partial}{\partial t}(\log F(t))=-\mathcal{R}ic.$
Derivation of $F^{2}(t)=\varphi_{t}^{*}(\tilde{F}^{2}(t))$ with respect to the parameter $t$, leads to
\begin{equation} \label{50}
\frac{\partial}{\partial t}(\log F(t))=\frac{1}{2}\frac{\frac{\partial}{\partial t}(\varphi_{t}^{*}(\tilde{F}^{2}(t)))}{\varphi_{t}^{*}(\tilde{F}^{2}(t))}.
\end{equation}
The term $\frac{\partial}{\partial t}(\varphi_{t}^{*}\tilde{F}^{2}(t))$ becomes
\begin{eqnarray} \label{51}
\frac{\partial}{\partial t}(\varphi_{t}^{*}\tilde{F}^{2}(t))&=&\frac{\partial}{\partial s}(\varphi_{s+t}^{*}(\tilde{F}^{2}(s+t)))\mid_{s=0}\\
&=&\varphi_{t}^{*}(\frac{\partial}{\partial t}\tilde{F}^{2}(t))+\frac{\partial}{\partial s}(\varphi_{s+t}^{*}(\tilde{F}^{2}(t)))\mid_{s=0}\nonumber \\
&=&\varphi_{t}^{*}(\frac{\partial}{\partial t}\tilde{F}^{2}(t))+\frac{\partial}{\partial s}((\varphi_{t}^{-1}\circ\hspace{0.1cm}\varphi_{t+s})^{*}(\varphi_{t}^{*}(\tilde{F}^{2}(t))))\mid_{s=0}\nonumber\\
&=&\varphi_{t}^{*}(\frac{\partial}{\partial t}\tilde{F}^{2}(t))+\mathcal{L}_{\frac{\partial}{\partial s}(\varphi_{t}^{-1}\circ\hspace{0.1cm}\varphi_{t+s})\mid_{s=0}}\varphi_{t}^{*}(\tilde{F}^{2}(t)).\nonumber
\end{eqnarray}
On the other hand, we have
\begin{eqnarray*} %\label{52}
\frac{\partial}{\partial s}(\varphi_{t}^{-1}\circ\hspace{0.1cm}\varphi_{t+s})\mid_{s=0}=(\varphi_{t}^{-1})_{*}((\frac{\partial}{\partial s}\varphi_{s+t})\mid_{s=0})=(\varphi_{t}^{-1})_{*}(\xi).
\end{eqnarray*}
Hence, (\ref{51}) is written
\begin{eqnarray*}%\label{53}
\frac{\partial}{\partial t}(\varphi_{t}^{*}\tilde{F}^{2}(t))=\varphi_{t}^{*}(\frac{\partial}{\partial t}\tilde{F}^{2}(t))+\mathcal{L}_{(\varphi_{t}^{-1})_{*}(\xi)}\varphi_{t}^{*}(\tilde{F}^{2}(t)).
\end{eqnarray*}
Replacing the last relation in (\ref{50}) and using the assumption (\ref{46}) we get
\begin{eqnarray*}%\label{54}
\frac{\partial}{\partial t}(\log F(t))&=&\frac{1}{2}\frac{\varphi_{t}^{*}(\frac{\partial}{\partial t}\tilde{F}^{2}(t))+\mathcal{L}_{(\varphi_{t}^{-1})_{*}(\xi)}
\varphi_{t}^{*}(\tilde{F}^{2}(t))}{\varphi_{t}^{*}(\tilde{F}^{2}(t))} \\
&=&\frac{1}{2}\frac{\varphi_{t}^{*}(-2\tilde{F}^{2}(t)\mathcal{R}ic(\tilde{F}(t))-\mathcal{L}_{\xi}\tilde{F}^{2}(t))+
\mathcal{L}_{(\varphi_{t}^{-1})_{*}(\xi)}\varphi_{t}^{*}(\tilde{F}^{2}(t))}{\varphi_{t}^{*}(\tilde{F}^{2}(t))}\\
&=&\frac{1}{2}\frac{\varphi_{t}^{*}(-2\tilde{F}^{2}(t)\mathcal{R}ic(\tilde{F}(t)))-\varphi_{t}^{*}(\mathcal{L}_{\xi}\tilde{F}^{2}(t)))
+\mathcal{L}_{(\varphi_{t}^{-1})_{*}(\xi)}\varphi_{t}^{*}(\tilde{F}^{2}(t))}{\varphi_{t}^{*}(\tilde{F}^{2}(t))}\\
&=&\frac{1}{2}\frac{-2\varphi_{t}^{*}(\tilde{F}^{2}(t))\varphi_{t}^{*}(\mathcal{R}ic(\tilde{F}(t)))}{\varphi_{t}^{*}(\tilde{F}^{2}(t))}.
\end{eqnarray*}
By virtue of Lemma \ref{lemmohem} we have
\begin{eqnarray*}
\frac{\partial}{\partial t}(\log F(t))
=-\varphi_{t}^{*}(\mathcal{R}ic(\tilde{F}(t)))
=-\mathcal{R}ic_{\varphi_{t}^{*}(\tilde{F}(t))}
=-\mathcal{R}ic_{F(t)}.
\end{eqnarray*}
Therefore, the Finsler structures $F(t)$ form a solution to the Finslerian Ricci flow. Hence the proof is complete.
\end{proof}
In the next step we assume that there is a solution to the Finslerian Ricci flow based on which we construct a solution to the Ricci-DeTurck flow in Finsler space.

\begin{prop}\label{main12}
Fix a compact Finsler surface $(M,\bar{F})$ with the related Finsler metric tensor $h$. Let $F(t)$, $t\in[0,T)$, be a family of solutions to the Ricci flow and $\varphi_{t}$ a one-parameter family of diffeomorphisms on $SM$ evolving under the following flow,
\begin{equation} \label{55}
\frac{\partial}{\partial t}\varphi_{t}=\Phi_{g(t),h}\varphi_{t}.\nonumber
\end{equation}
Then the Finsler structures $\tilde{F}(t)$ defined by $F^{2}(t)=\varphi^{*}_{t}(\tilde{F}^{2}(t))$ form a solution to the following Ricci-DeTurck flow
\begin{equation} \label{56}
\frac{\partial}{\partial t}\tilde{F}^{2}(t)=-2\tilde{F}^{2}(t)\mathcal{R}ic(\tilde{g}(t))-\mathcal{L}_{\xi}\tilde{F}^{2}(t),\nonumber
\end{equation}
where, $\xi=(\Phi_{\tilde{g}(t),h}Id)^{i}\frac{\partial}{\partial x^i}$ and $\tilde{g}(t)$ is the metric tensor related to $\tilde{F}(t)$. Furthermore, for all $z\in SM$ and $t\in[0,T)$ we have
\begin{equation*}% \label{57}
\frac{\partial}{\partial t}\varphi_{t}(z)=\xi(\varphi_{t}(z),t).
\end{equation*}
%where, $\tilde{\varphi}_{t}$ is the canonical lift of $\varphi_{t}$ on $SM$.
\end{prop}
\begin{proof}
Using Lemma \ref{main2} we have
\begin{align*}
\frac{\partial}{\partial t}\varphi_{t}&=\Phi_{g(t),h}\varphi_{t}=\Phi_{\varphi^{*}_{t}
(\tilde{g}(t)),h}\varphi_{t}=\Phi_{\varphi^{*}_{t}
(\tilde{g}(t)),h} Id \circ\varphi_{t}\\
&=\Phi_{\varphi^{*}_{t}
(\tilde{g}(t)),h}\varphi_{t}^{*} Id=\Phi_{\tilde{g}(t),h}Id=\xi,
\end{align*}
for all $z\in SM$ and $t\in[0,T)$. Using $F^{2}(t)=\varphi^{*}_{t}(\tilde{F}^{2}(t))$ leads to
\begin{align}\label{59+1}
\frac{\partial}{\partial t}(\log F(t))&=
\frac{1}{2}\frac{\frac{\partial}{\partial t}(\varphi^{*}_{t}(\tilde{F}^{2}(t)))}{\varphi^{*}_{t}(\tilde{F}^{2}(t))}\nonumber\\
&=\frac{1}{2}\frac{\varphi^{*}_{t}(\frac{\partial}{\partial t}\tilde{F}^{2}(t))+\mathcal{L}_{(\varphi^{-1}_{t})_{*}(\xi)}
\varphi^{*}_{t}(\tilde{F}^{2}(t))}{\varphi^{*}_{t}(\tilde{F}^{2}(t))}\nonumber\\
&=\frac{1}{2}\frac{\varphi^{*}_{t}(\frac{\partial}{\partial t}\tilde{F}^{2}(t)+\mathcal{L}_{\xi}\tilde{F}^{2}(t))}{\varphi^{*}_{t}(\tilde{F}^{2}(t))}.
\end{align}
By assumption, $F(t)$ form a solution to the Finslerian Ricci flow (\ref{20})
\begin{equation} \label{60}
0=\frac{\partial}{\partial t}(\log F(t))+\mathcal{R}ic_{F(t)}.
\end{equation}
Thus by means of (\ref{59+1}), (\ref{60}) and Lemma \ref{lemmohem} we have
\begin{align*}
0&=\frac{\varphi^{*}_{t}(\frac{\partial}{\partial t}\tilde{F}^{2}(t)+\mathcal{L}_{\xi}\tilde{F}^{2}(t))}{\varphi^{*}_{t}(\tilde{F}^{2}(t))}+2\mathcal{R}ic_{F(t)}\nonumber\\
&=\frac{\varphi^{*}_{t}(\frac{\partial}{\partial t}\tilde{F}^{2}(t)+\mathcal{L}_{\xi}\tilde{F}^{2}(t))}{\varphi^{*}_{t}(\tilde{F}^{2}(t))}
+2\mathcal{R}ic_{\varphi^{*}_{t}(\tilde{F}(t))}\nonumber\\
&=\frac{\varphi^{*}_{t}(\frac{\partial}{\partial t}\tilde{F}^{2}(t)+\mathcal{L}_{\xi}\tilde{F}^{2}(t))}{\varphi^{*}_{t}(\tilde{F}^{2}(t))}
+2\varphi^{*}_{t}(\mathcal{R}ic_{\tilde{F}(t)})\nonumber\\
&=\frac{\varphi^{*}_{t}(\frac{\partial}{\partial t}\tilde{F}^{2}(t)+\mathcal{L}_{\xi}\tilde{F}^{2}(t))+2
\varphi^{*}_{t}(\tilde{F}^{2}(t))\varphi^{*}_{t}(\mathcal{R}ic_{\tilde{F}(t)})}{\varphi^{*}_{t}
(\tilde{F}^{2}(t))}\nonumber\\
&=\frac{\varphi^{*}_{t}(\frac{\partial}{\partial t}\tilde{F}^{2}(t)+\mathcal{L}_{\xi}\tilde{F}^{2}(t)+2
\tilde{F}^{2}(t)\mathcal{R}ic_{\tilde{F}(t)})}{\varphi^{*}_{t}
(\tilde{F}^{2}(t))}\nonumber.
\end{align*}
Therefore,
$
\varphi^{*}_{t}(\frac{\partial}{\partial t}\tilde{F}^{2}(t)+\mathcal{L}_{\xi}\tilde{F}^{2}(t)+2
\tilde{F}^{2}(t)\mathcal{R}ic_{\tilde{F}(t)})=0.\nonumber
$
This implies
\begin{equation}
\frac{\partial}{\partial t}\tilde{F}^{2}(t)=-2
\tilde{F}^{2}(t)\mathcal{R}ic_{\tilde{F}(t)}-\mathcal{L}_{\xi}\tilde{F}^{2}(t).\nonumber
\end{equation}
Therefore, $\tilde{F}(t)$ is a solution to the Ricci-DeTurck flow, as we have claimed.
\end{proof}
{\bf Proof of Theorem \ref{main14}.} In order to check the existence statement, recall that by means of Theorem \ref{main8}, there exists a solution $\tilde{F}(t)$ to the Finslerian Ricci-DeTurck flow (\ref{22}) which is defined on some time interval $[0,T)$ and satisfies $\tilde{F}(0)=F_{0}$. Let $\varphi_{t}$ be the solution of the ODE
\begin{eqnarray*}%\label{47}
\frac{\partial}{\partial t}\varphi_{t}(z)=(\Phi_{\tilde{g}(t),h}Id)(\varphi_{t}(z),t)=\xi(\varphi_{t}(z),t),\nonumber
\end{eqnarray*}
with the initial condition $\varphi_{0}(z)=z$, for $z\in SM$ and $t\in[0,T)$. By Proposition \ref{main11}, the Finsler structures $F^{2}(t)=\varphi^{*}_{t}(\tilde{F}^{2}(t))$ form a solution to the Finslerian Ricci flow (\ref{20}) with $F(0)=F_{0}$. This completes the existence statement.

For uniqueness statement assume that $F_{1}(t)$ and $F_{2}(t)$ are both solutions to the Finslerian Ricci flow defined on some time interval $[0,T)$ and satisfy $F_{1}(0)=F_{2}(0)$. We claim $F_{1}(t)=F_{2}(t)$ for all $t\in[0,T)$. In order to prove this fact, we argue by contradiction. Suppose that $F_{1}(t)\neq F_{2}(t)$ for some $t\in[0,T)$. Let's consider a real number $\tau\in[0,T)$ where $\tau=\inf\{t\in[0,T):F_{1}(t)\neq F_{2}(t)\}$. Clearly, $F_{1}(\tau)=F_{2}(\tau)$. Let $\varphi^{1}_{t}$ be a solution of the flow
\begin{equation} \label{67}
\frac{\partial}{\partial t}\varphi^{1}_{t}=\Phi_{g_{1}(t),h}\varphi^{1}_{t},\nonumber
\end{equation}
with initial condition $\varphi^{1}_{\tau}=Id$ and $\varphi^{2}_{t}$ a solution of the flow
\begin{equation}
\frac{\partial}{\partial t}\varphi^{2}_{t}=\Phi_{g_{2}(t),h}\varphi^{2}_{t},\nonumber
\end{equation}
with initial condition $\varphi^{2}_{\tau}=Id$. It follows from the standard theory of parabolic differential equations that $\varphi^{1}_{t}$ and $\varphi^{2}_{t}$ are defined on some time interval $[\tau,\tau+\epsilon)$, where, $\epsilon$ is a positive real number. Moreover, if we choose $\epsilon>0$ small enough, then $\varphi^{1}_{t}$ and $\varphi^{2}_{t}$ are diffeomorphisms for all $t\in[\tau,\tau+\epsilon)$. For each $t\in[\tau,\tau+\epsilon)$ we define two Finsler structures $\tilde{F}_{1}(t)$ and $\tilde{F}_{2}(t)$ by $(F_{1}(t))^{2}=(\varphi^{1}_{t})^{*}(\tilde{F}_{1}(t))^{2}$ and $(F_{2}(t))^{2}=(\varphi^{2}_{t})^{*}(\tilde{F}_{2}(t))^{2}$. It follows from Proposition \ref{main12} that $\tilde{F}_{1}(t)$ and $\tilde{F}_{2}(t)$ are solutions of the Finslerian Ricci-DeTurck flow. Since $\tilde{F}_{1}(\tau)=\tilde{F}_{2}(\tau)$, the uniqueness statement in Theorem \ref{main8} implies that $\tilde{F}_{1}(t)=\tilde{F}_{2}(t)$ for all $t\in[\tau,\tau+\epsilon)$. For each $t\in[\tau,\tau+\epsilon)$, we define a vector field $\xi$ on $SM$ by
\begin{equation}
\xi=\Phi_{\tilde{g}_{1}(t),h}Id=\Phi_{\tilde{g}_{2}(t),h}Id.\nonumber
\end{equation}
By Proposition \ref{main12}, we have
\begin{eqnarray*}%\label{47}
\frac{\partial}{\partial t}\varphi^{1}_{t}(z)=\xi(\varphi^{1}_{t}(z),t),\nonumber
\end{eqnarray*}
and
\begin{eqnarray*}%\label{47}
\frac{\partial}{\partial t}\varphi^{2}_{t}(z)=\xi(\varphi^{2}_{t}(z),t),\nonumber
\end{eqnarray*}
for $z\in SM$ and $t\in[\tau,\tau+\epsilon)$. Since $\varphi^{1}_{\tau}=\varphi^{2}_{\tau}=Id$, it follows that $\varphi^{1}_{t}=\varphi^{2}_{t}$ for all $t\in[\tau,\tau+\epsilon)$.  Putting these facts together, we conclude that
\begin{equation*}
(F_{1}(t))^{2}=(\varphi^{1}_{t})^{*}(\tilde{F}_{1}(t))^{2}=(\varphi^{2}_{t})^{*}(\tilde{F}_{2}(t))^{2}=
(F_{2}(t))^{2},
\end{equation*}
for all $t\in[\tau,\tau+\epsilon)$. Therefore, $F_{1}(t)=F_{2}(t)$ for all $t\in[\tau,\tau+\epsilon)$. This contradicts the definition of $\tau$. Thus the uniqueness holds well. This completes the proof of Theorem \ref{main14}.\hspace{\stretch{1}}$\Box$
\begin{ex}
Let $(M,F_{0})$ be a compact Finsler surface.
We are going to obtain a solution to the Ricci flow \eqref{20}.
It's well known in dimension two, a Finsler metric is of isotropic Ricci scalar (or Einstein)  if and only if it is of isotropic flag curvature.
  Therefore, $F_{0}$ is an Einstein metric and we have $\mathcal{R}ic_{F_{0}}=K$ where, $K=K(x)$ is a scalar function on $M$. Consider a family of scalars $\tau(t)$ defined by
\begin{equation*}
\tau(t):=1-2Kt>0.
\end{equation*}
Define a smooth one-parameter family of Finsler structures on $M$ by
\begin{equation*}
F^{2}(t):=\tau(t)F_{0}^2.
\end{equation*}
Thus we have
\begin{equation*}
\log(F(t))=\frac{1}{2}\log(\tau(t)F_{0}^2).
\end{equation*}
Derivative with respect to $t$ yields
\begin{equation}\label{exp1}
\frac{\partial}{\partial t}\log(F(t))=-\frac{K}{\tau(t)}=-\frac{\mathcal{R}ic_{F_{0}}}{\tau(t)}.
\end{equation}
%\begin{equation*}
%\frac{\partial}{\partial t}\log(F(t))=-\frac{\mathcal{R}ic_{F_{0}}}{\tau(t)}.
%\end{equation*}
On the other hand, by straight forward computations we have $\frac{1}{\tau(t)}\mathcal{R}ic_{F_{0}}=\mathcal{R}ic_{\tau(t)^{\frac{1}{2}}F_{0}}$, for more details see \cite[p.\ 926]{BY}. Replacing the last relation in (\ref{exp1}) leads to
\begin{equation*}
\frac{\partial}{\partial t}\log(F(t))=-\mathcal{R}ic_{\tau(t)^{\frac{1}{2}}F_{0}}=-\mathcal{R}ic_{{F(t)}}.
\end{equation*}
Hence, $F(t)$ is a solution to the Ricci flow equation \eqref{20}.
\end{ex}
\begin{ex}
Let $F(t)$ be a family of Finsler structures on the sphere $\mathbb{S}^2$ defined by $ F^2(t)=a_{ij}(t)y^iy^j$ where, $a_{ij}(t)$ is a well known Riemannian metric on $\R^2$, called  the Rosenau metric
$$ a_{ij}(t)=\frac{8\sinh (-t)}{1+2 \cosh (-t) |x|^2+|x|^4}\delta_{ij} , \quad t\in (-\infty , 0) , x\in \R^2. $$
It is well known that $a_{ij}$ extends to a metric on $\mathbb{S}^2$. The related Finsler metric tensor of $F(t)$ is
\begin{equation*}
g_{ij}(t):=(\frac{1}{2}F^{2})_{y^iy^j}=a_{ij}(t).
\end{equation*}
By straight forward computations, $R(a(t))$ the scalar curvature of the Riemannian metric $a_{ij}(t)$ is
\begin{equation*}
R(a(t))=\frac{\cosh(-t)}{\sinh(-t)}-\frac{2\sinh(-t)|x|^2}{1+2\cosh(-t)|x|^2+|x|^4}.
\end{equation*}
The Ricci tensor of $g_{ij}$ and $a_{ij}$ coincides. Hence
\begin{equation*}
Ric_{ij}(g(t))=Ric_{ij}(a(t))=\frac{1}{2}R(a(t))a_{ij}(t).
\end{equation*}
 From $F(t)=(a_{ij}(t)y^iy^j)^{\frac{1}{2}}$, we have
\begin{equation*}
\log(F(t))=\frac{1}{2}\log(a_{ij}(t)y^iy^j).
\end{equation*}
Derivative with respect to $t$ leads to
\begin{equation}\label{exp}
\frac{\partial}{\partial t}\log F(t)=\frac{1}{2}\frac{\partial}{\partial t}(a_{ij}(t))l^il^j,
\end{equation}
where,
\begin{equation*}
\frac{\partial}{\partial t}(a_{ij}(t))={\Big (}\frac{-8\cosh (-t)}{1+2 \cosh (-t) |x|^2+|x|^4}+\frac{16\sinh ^2(-t)|x|^2}{(1+2 \cosh (-t) |x|^2+|x|^4)^2}{\Big )}\delta_{ij}.
\end{equation*}
On the other hand, we have
\begin{align*}
\mathcal{R}ic(g(t))&=l^il^jRic_{ij}(g(t))=\frac{1}{2}l^il^jR(a(t))a_{ij}(t)\\
&=\frac{1}{2}l^il^j{\Big (}\frac{8\cosh (-t)}{1+2 \cosh (-t) |x|^2+|x|^4}-\frac{16\sinh ^2(-t)|x|^2}{(1+2 \cosh (-t) |x|^2+|x|^4)^2}{\Big )}\delta_{ij}.
\end{align*}
Comparing the last equation and (\ref{exp}) we have
\begin{equation*}
\frac{\partial}{\partial t}\log F(t)=-\mathcal{R}ic(g(t)).
\end{equation*}
Consequently, F(t) form a solution to the Finsler  Ricci flow \eqref{20} on $\mathbb{S}^2$.
\end{ex}
%%%%%%%%%%%%%%%%%%%%%%%%%
{\bf Acknowledgment}\\
The authors would like to thank Prof. David Bao for his valuable comments.
%This work is partially supported by INSF, under the grant 95002579.
This work is partially supported by Iran National Science Foundation (INSF), under the grant 95002579.
%%%%%%%%%%%%%%%%%%%%%%%%%%%%%%%%%%
 
%\bibliography{mybibfile}

\begin{thebibliography} {WWW99}
\bibitem{HAZ} H. Akbar-Zadeh,\emph{ Initiation to global Finslerian geometry}, Vol. \textbf{68}. Elsevier Science, 2006.
\bibitem{Bao} D. Bao, \emph{On two curvature-driven problems in Riemann-Finsler geometry}: In memory of Makoto Matsumoto, Advanced studies in pure mathematics, Vol. \textbf{48}, Mathematical Society, Japan, Tokyo,  (2007), 19-71.
\bibitem{BCS} D. Bao, S. Chern, Z. Shen, \emph{An introduction to Riemann-Finsler Geometry}. Graduate Texts in Mathematics, Vol. \textbf{200}, Springer, 2000.
%\bibitem{BAOS} D. Bao and Z. Shen, \emph{On the volume of unit tangent spheres in a Finsler manifold}, Results in Mathematics, Vol. \textbf{26}, (1994), pp. 1-17.
\bibitem{JB}  B. Bidabad, P. Joharinad, \textit{Conformal vector fields on Finsler spaces}, Differential Geometry and its Applications, \textbf{31}, (2013), 33--40.
\bibitem{YB2} B. Bidabad, M. Yar Ahmadi, \textit{Convergence of Finslerian metrics under Ricci flow}, Sci. China. Math, Vol. {\bf 59}, (2016), 741-750.
%\bibitem{YB1} B. Bidabad, M. Yar Ahmadi, \emph{On compact Ricci solitons in Finsler geometry}, C. R. Acad. Sci. Paris, Ser I, Vol. \textbf{353}, (2015),  1023-1027.
\bibitem{BY} B. Bidabad, M. Yar Ahmadi, \emph{On quasi-Einstein Finler spaces}, Bull. Iranian Math. Soc., Vol. \textbf{40}, No. 4, (2014),  921-930.
%\bibitem{BM} I. Bucataru, R. Miron, \textit{Finsler-Lagrange Geometry. Applications to dynamical systems}, Editura Academiei Romane, Bucures¸ti, 2007.
\bibitem{Chow1} B. Chow, D. Knopf, \textit{The Ricci Flow: An Introduction,} Mathematical Survays and Monographs, Vol. \textbf{110}, AMS,  Providence, RI, 2004.
\bibitem{Chow2} B. Chow, \textit{The Ricci flow on the 2-sphere}, J. Differential Geom. \textbf{33}, No. 2, (1991), 325-334.
\bibitem{DeT} D. M. DeTurck, \textit{Deforming metrics in the direction of their Ricci tensors}, J. Differential Geom. \textbf{18}(1) (1983), 157--162.
\bibitem{Ham1} R. S. Hamilton, \textit{Three manifolds with positive Ricci curvature}, J. Differential Geom. \textbf{17}, (1982), 255-306.
\bibitem{Ham2} R. S. Hamilton, \textit{The Ricci flow on surfaces}, Math. and General Relativity, Contemporary Math. \textbf{71}, (1988), 237-262.
\bibitem{Rog} M. Renardy, R. C. Rogers, \textit{An introduction to partial differential equations}, Texts in applied mathematics, second edition, Springer 2004.
\end{thebibliography}
%%%%%%%%%%%%%%%%%%%%%%%%%%%
{\small   Behroz Bidabad, bidabad@aut.ac.ir}\\
{\small  Maral Khadem Sedaghat,
m\_sedaghat@aut.ac.ir \\
 Faculty of Mathematics and Computer Science, \\ Amirkabir University of Technology (Tehran Polytechnic),\\ Hafez Ave., 15914 Tehran, Iran.}

\end{document}